\documentclass[12pt,english,reqno]{amsart}
\usepackage[T1]{fontenc}
\usepackage[utf8]{inputenc}
\usepackage[a4paper]{geometry}
\usepackage{babel}
\usepackage{units}
\usepackage{enumitem}
\usepackage{amstext}
\usepackage{amsthm}
\usepackage{amssymb}
\usepackage{setspace}
\makeatletter

\providecommand{\tabularnewline}{\\}

\numberwithin{equation}{section}
\numberwithin{figure}{section}
\theoremstyle{plain}
\newtheorem{thm}{\protect\theoremname}[section]
\theoremstyle{definition}
\newtheorem{example}[thm]{\protect\examplename}
\theoremstyle{definition}
\newtheorem{defn}[thm]{\protect\definitionname}
\theoremstyle{remark}
\newtheorem{observation}[thm]{\protect\observationname}
\theoremstyle{plain}
\newtheorem{prop}[thm]{\protect\propositionname}
\theoremstyle{remark}
\newtheorem{claim}[thm]{\protect\claimname}
\theoremstyle{plain}
\newtheorem{cor}[thm]{\protect\corollaryname}
\theoremstyle{plain}
\newtheorem{lem}[thm]{\protect\lemmaname}
\theoremstyle{plain}
\newtheorem{fact}[thm]{\protect\factname}
\theoremstyle{remark}
\newtheorem{hypothesis}[thm]{\protect\hypothesisname}
\theoremstyle{remark}
\newtheorem{rem}[thm]{\protect\remarkname}
\newlist{casenv}{enumerate}{4}
\setlist[casenv]{leftmargin=*,align=left,widest={iiii}}
\setlist[casenv,1]{label={{\itshape\ \casename} \arabic*.},ref=\arabic*}
\setlist[casenv,2]{label={{\itshape\ \casename} \roman*.},ref=\roman*}
\setlist[casenv,3]{label={{\itshape\ \casename\ \alph*.}},ref=\alph*}
\setlist[casenv,4]{label={{\itshape\ \casename} \arabic*.},ref=\arabic*}

\usepackage{amsmath}
\usepackage{bbm}
\usepackage{dsfont}
\usepackage[dvipsnames]{xcolor}

\usepackage[unicode=true,pdfusetitle,
 bookmarks=true,bookmarksnumbered=false,bookmarksopen=false,
 breaklinks=false,pdfborder={0 0 0},pdfborderstyle={},backref=false,colorlinks=false]
 {hyperref}

\allowdisplaybreaks
\usepackage{tikz}
\usetikzlibrary{arrows,decorations.markings,shapes,patterns}

\usepackage{etoolbox}
\patchcmd{\@settitle}{\uppercasenonmath\@title}{}{}{}
\patchcmd{\@setauthors}{\MakeUppercase}{}{}{}
\patchcmd{\section}{\scshape}{}{}{}

\makeatother

\providecommand{\casename}{Case}
\providecommand{\claimname}{Claim}
\providecommand{\corollaryname}{Corollary}
\providecommand{\definitionname}{Definition}
\providecommand{\examplename}{Example}
\providecommand{\factname}{Fact}
\providecommand{\hypothesisname}{Hypothesis}
\providecommand{\lemmaname}{Lemma}
\providecommand{\observationname}{Observation}
\providecommand{\propositionname}{Proposition}
\providecommand{\remarkname}{Remark}
\providecommand{\theoremname}{Theorem}

\begin{document}
\global\long\def\One{\mathds{1}}%

\global\long\def\Laplacian{\Delta}%

\global\long\def\grad{\nabla}%

\global\long\def\norm#1{\left\Vert #1\right\Vert }%

\global\long\def\Z{\mathbb{Z}}%

\global\long\def\R{\mathbb{R}}%

\global\long\def\N{\mathbb{N}}%

\global\long\def\P{\mathbb{P}}%

\global\long\def\E{\mathbb{E}}%

\global\long\def\cL{\mathcal{L}}%

\global\long\def\cD{\mathcal{D}}%

\global\long\def\cA{\mathcal{A}}%

\global\long\def\cB{\mathcal{B}}%

\global\long\def\floor#1{\left\lfloor #1\right\rfloor }%

\global\long\def\ceil#1{\left\lceil #1\right\rceil }%

\global\long\def\var{\operatorname{Var}}%

\global\long\def\cov{\operatorname{Cov}}%

\global\long\def\dd{\operatorname{d}}%

\global\long\def\MC{\mathcal{C}}%

\global\long\def\MCsize{N}%

\global\long\def\MClen{l}%

\global\long\def\boxsize{\lambda}%

\global\long\def\Dom{\operatorname{Dom}}%

\global\long\def\Loss{\operatorname{Loss}}%

\global\long\def\EBarrier{\operatorname{EB}}%

\global\long\def\Dirichlet{\mathcal{D}}%

\title{Noncooperative models of kinetically constrained lattice gases}
\author{Assaf Shapira}
\address{MAP5, Université Paris Cité}
\email{\href{mailto:assaf.shapira@normalesup.org}{assaf.shapira@normalesup.org}}
\urladdr{\href{http://assafshap.github.io}{assafshap.github.io}}
\begin{abstract}
We study a family of conservative interacting particle systems with
degenerate rates called noncooperative kinetically constrained lattice
gases. We prove for all models in this family the diffusive scaling of
the relaxation time, the positivity of the diffusion coefficient,
and the positivity of the self-diffusion coefficient.
\end{abstract}

\maketitle

\section{Introduction}

Kinetically constrained lattice gases are interacting particle systems
introduced by physicists in order to better understand glassy materials
(see, e.g., \cite{KobAndersen,RitortSollich}). The basic underlying
hypothesis behind these models is that glassy behavior is a dynamic
effect, and the role of interactions is irrelevant. Under this hypothesis,
we can explain why glasses are rigid using the \emph{cage effect}---even
though their microscopic structure is amorphous, glasses at low temperatures
have a very high density, and molecules are unable to move since they
are blocked by neighboring molecules.

In order to model this effect, we consider the lattice $\Z^{d}$,
describing a coarse graining of the glass. Each site, corresponding
to some region in the glass, could be either \emph{occupied} or \emph{empty},
representing dense or dilute zones. We think of the glass as very
dense, so the small parameter $q$ will be the ratio of \emph{empty}
sites.

The dynamics of kinetically constrained lattice gases is conservative---particles
could jump between neighbors, turning an occupied site empty and a
neighboring empty site occupied. However, not all jumps are allowed---in
order to imitate the cage effect, when the local neighborhood of a
particle is too dense it is blocked. That is, particles are only
able to move under a certain constraint, satisfied when there are
many vacancies nearby. Different kinetically constraint lattice gases
are given by different choices of this constraint, namely, different
interpretations of the neighborhood being ``too dense''.

To fix
an idea, consider a one dimensional model introduced in \cite{BT04}
(see Example \ref{exa:1d}), where a particle is allowed jump to an
empty neighbor, if it has at least two empty neighbors before or after
the jump (including the site it jumped to/from). Note that if a particle
is allowed to jump, than it is also allowed to jump back immediately
after. This is a property we require for all kinetically constrained
models, and it guarantees a noninteracting equilibrium.

It is instructive to compare these models to another family of interacting
particle systems, the nonconservative kinetically constrained models
(see, e.g., \cite{GarrahanSollichToninelli2011}). In those models,
rather than jumping between sites, particles appear and disappear
under the constraint. These models are in general simpler to analyze,
and, at least in one and two dimensions, we have a relatively good
understanding of their behavior \cite{MMT19KCMUniversality,MMT20Duarte,HMT21Universality1,HMT20Universality2,HM22Universality,Hartarsky21Universality}.
In fact, one can identify a handful of universality classes describing
the properties of a kinetically constrained model. Moreover, a simple
criterion allows us to determine, given any translation invariant
local constraint, to which universality class the model belongs. In
the case of conservative kinetically constrained lattice gases, however,
only a few specific models have been analyzed \cite{BT04,CMRT2010,GoncalvesLandimToninelli,Nagahata12,ToninelliBiroliFisher,MST19KA,BlondelToninell2018KAtagged,ES20KAtagged,S23KA_HL},
and no general results are available.

We distinguish between two types of kinetically constrained lattice
gases---cooperative and noncooperative. In a cooperative dynamics,
any large scale change in the configuration forces many particles to move in order to "free up" space. 
In noncooperative models, small empty clusters can move around
the lattice, without requiring any cooperation from other sites near
them. Consider the example introduced above. Figure \ref{fig:mc_1d}
shows how, in two allowed transitions, two neighboring vacancies can
propagate to the right, no matter what the occupation is elsewhere.
We say that these vacancies form a \emph{mobile cluste}r. Noncooperative
models are those where a mobile cluster exists, and cooperative models
are models where no finite set of vacancies can propagate without
any outside help. See Definition \ref{def:mobile_cluster_noncooperative}.

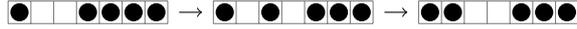
\begin{figure}
\begin{tikzpicture}[scale=0.3, every node/.style={scale=0.6}]
	\draw[step=1,gray] (0,0) grid +(7,1);
	\fill[black,shift={(0.5,0.5)}] (0,0) circle (0.4);
	\foreach \x in {3,...,6}{
		\fill[black,shift={(0.5,0.5)}] (\x,0) circle (0.4);
	}
	
	\draw[->,shift={(0.5,0.5)}]  (7,0) to (8,0);
	
	\draw[step=1,gray] (9,0) grid +(7,1);
	\fill[black,shift={(0.5,0.5)}] (9,0) circle (0.4);
	\fill[black,shift={(0.5,0.5)}] (11,0) circle (0.4);
	\foreach \x in {13,14,15}{
		\fill[black,shift={(0.5,0.5)}] (\x,0) circle (0.4);
	}
	
	\draw[->,shift={(0.5,0.5)}]  (16,0) to (17,0);
	
	\draw[step=1,gray] (18,0) grid +(7,1);
	\fill[black,shift={(0.5,0.5)}] (18,0) circle (0.4);
	\fill[black,shift={(0.5,0.5)}] (19,0) circle (0.4);
	\foreach \x in {22,23,24}{
		\fill[black,shift={(0.5,0.5)}] (\x,0) circle (0.4);
	}
\end{tikzpicture}

\caption{\label{fig:mc_1d}This figure shows how, in the model described in
Example \ref{exa:1d}, a mobile cluster can propagate. The mobile
cluster here consists of the two empty sites, and after a sequence
of $1$ allowed transitions it is moved one step to the right. See
Example \ref{exa:tr_1d}.}
\end{figure}

One simple implication of the presence of a mobile cluster is that
the critical density of the model is $1$ (equivalently, the critical
value of $q$ is $0$). This means that for any $q>0$, in an infinite
system, there exists with probability $1$ a sequence of allowed transitions
in the end of which the origin (or any other arbitrary vertex) is
empty. Indeed, since a mobile cluster consists of some fixed number
of vacancies, if $q>0$ there will almost surely be an empty mobile cluster
somewhere in $\Z^{d}$. We can then move this cluster until one of
its vacancies reaches the origin. In cooperative models identifying
the critical density is more complicated. The only cooperative kinetically
constrained lattice gas studied in the mathematics literature is the
Kob-Andersen model \cite{ToninelliBiroliFisher}, where the critical
density is also $1$; but in general cooperative models may have critical
densities which are strictly smaller.

Close to criticality, when $q\ll1$, most sites are occupied, and
the constraint is rarely satisfied. The dynamics then tends to slow
down, making typical time scales diverge. We will try to understand 
how significant this effect is. In the unconstrained model
(namely the simple exclusion process), time scales diffusively, as
the square of the distance:
\[
\text{typical time}\approx C\times\text{ typical distance}^{2}.
\]
 We will see that noncooperative models are also diffusive---the
constraint may affect the coefficient $C$, but the exponent remains
$2$. This will be done in four different contexts, giving different
interpretations to ``typical time'' and ``typical distance''.

The first time scale we study is the \emph{relaxation time}, describing
the time scale over which correlations are lost. Consider some observable
$f$ depending on the configuration, and measure it at time $0$ and
at time $t$. In some cases, the correlation between these two quantities,
$f_{0}$ and $f_{t}$, decreases exponentially fast with $t$---
\[
\text{Corr}(f_{0},f_{t})\le e^{-t/\tau}.
\]
The best (i.e. smallest) coefficient $\tau$ for which this decay
hold \emph{uniformly} (i.e. for all $f$) is the relaxation time.
In general, the relaxation time can be infinite, and this is in fact
the case for kinetically contrained lattice gases on the infinite
lattice. In sections \ref{sec:relaxation_reservoir} and \ref{sec:relaxation_closed}
we study the relaxation time on a \emph{finite} box, of length $L$.
We will see that the relaxation time is proportional to $L^{2}$,
and that the corresponding coefficient diverges as a power law for
small values of $q$.

In Section \ref{sec:diffusion} we study the diffusion coefficient
associated with the dynamics. This coefficient, generally speaking,
describes the large scale evolution of the density profile. Consider
for example a one dimensional model defined on a large interval $\{1,\dots,L\}$.
Assume that the initial configuration approximates some given density
profile $\rho_{0}:[0,1]\to[0,1]$. Roughly speaking, this means that
the number of particles in an interval $\{x-l/2,\dots,x,\dots,x+l/2\}$
of ``medium'' length (i.e. $1\ll l\ll L$) is close to $l\rho(x/L)$.
Then, when the system is diffusive, we expect the configuration at
a later time $t$ to approximate the same profile $\rho_{0}$ if $t\ll L^{2}$
(before the diffusive time scale), some evolving profile $\rho(t/L^{2},\cdot)$
when $t$ is of the order $L^{2}$ (in the diffusive time scale),
and a uniform profile when $t\gg L^{2}$ (after the diffusive time
scale). Moreover, the evolution in the diffusive scale is given by
a \emph{diffusion equation} 
\[
\partial_{\tau}\rho(\tau,\xi)=\partial_{\xi}\,D(\rho(\tau,\xi))\partial_{\xi}\rho(\tau,\xi).
\]
The \emph{diffusion coefficient} $D$ tells us, within the diffusive
scale, how fast the density profile changes. In particular, if $D=0$
the density profile does not evolve in diffusive time scales. When
this picture indeed describes the behavior of the model, we say that
it converges to a hydrodynamic limit in the diffusive scale. This
hydrodynamic limit is given by the diffusion equation above. For a more 
complete discussion see, e.g., \cite{KipnisLandim}.

Proving rigorously converges to a hydrodynamic limit is not an easy
task, accomplished only for one example of a kinetically constrained
lattice gas \cite{GoncalvesLandimToninelli,BlondelGoncalvesSimon}.
In fact, it cannot hold in full generality---a configuration such
as the one shown in Figure \ref{fig:blocked_configuration} approximates
the profile 
\[
\rho_{0}(x)=\begin{cases}
1 & x\le L/2,\\
2/3 & x>L/2.
\end{cases}
\]
At the same time, the configuration is blocked, namely, no particle
is allowed to jump. Thus, the density profile remains fixed, and cannot
converge to a hydrodynamic limit. This initial configuration, though,
is very specific, and one may still hope that, by restricting to a more
generic initial state, the dynamics will convergence to a hydrodynamic
limit. This is proven in \cite{GoncalvesLandimToninelli} for the
model they study, but a general proof seems to be very difficult.

\begin{figure}
\begin{tikzpicture}[scale=0.3, every node/.style={scale=0.6}]
	\draw[step=1,gray] (0,0) grid +(18,1);
	
	\foreach \x in {0,...,8}{
		\fill[black,shift={(0.5,0.5)}] (\x,0) circle (0.4);
	}
	\draw[BrickRed,thick] (9,-0.2)--(9,1.2);
	\foreach \x in {10,11,13,14,16,17}{
		\fill[black,shift={(0.5,0.5)}] (\x,0) circle (0.4);
	}
\end{tikzpicture}

\caption{\label{fig:blocked_configuration}We see here a blocked configuration
for the model in Example \ref{exa:1d}---in the left half all sites
are filled, while in the right half one in every three sites is empty.
No particle could jump to an empty site, hence the configuration is
blocked. In particular, it cannot converge to the hydrodynamic limit.}
\end{figure}
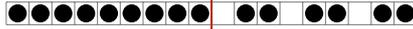

Still, even without proving convergence, studying
the diffusion coefficient is an interesting problem, allowing us to
obtain a plausible candidate for the hydrodynamic limit \cite{AritaKrapivskyMallick2018,TeomyShokef2017KAHL,S23KA_HL}.
Moreover, the strategy of \cite{S23KA_HL} shows convergence to a
hydrodynamic limit in a ``soft'' sense whenever the model is rotation
invariant. In particular, a strictly positive diffusion coefficient
is a good indication that the density profile evolves over diffusive
time scales. In Section \ref{sec:diffusion} we show that the diffusion
coefficient of noncooperative kinetically constrained lattice gases
is indeed positive, and that it decays at most polynomially fast for
small $q$.

The last interpretation of ``typical time'' and ``typical distance''
we consider is perhaps the most intuitive. Assume that the initial
configuration has a particle at the origin called the \emph{tracer}
(but otherwise sampled from equilibrium). One may think of the tracer
as playing the role of the pollen grain in Brown's famous experiment.
We then follow its motion, and ask what is the time it would typically
take in order to cross a certain distance. Diffusive scaling means
that this time scales as the square of the distance. A general argument
of \cite{KV} shows a much stronger result---under diffusive scaling,
the path of the tracer converges to a Brownian motion. The variance
of this Brownian motion is called the \emph{self diffusion} $D_{s}$,
and when it is strictly positive the Brownian motion in nondegenerate,
i.e., the relevant time scale is indeed diffusive.

All quantities mentioned above have variational characterizations,
involving infima or suprema over local functions, see equations (\ref{eq:t_rel}),
(\ref{eq:D_variational}), and (\ref{eq:D_s}). These formulations
allow us to analyze them using canonical path arguments, which in
the lack of attractivity have proven extremely useful in the study
of kinetically constrained models and kinetically constrained lattice
gases (see, e.g., \cite{CMRT08,CMRT2010,BT04,MST19KA,BlondelToninell2018KAtagged,S23KA_HL}).
In this paper, following \cite{MST19KA,ES20KAtagged,S23KA_HL}, we
formulate these argument in the language of \emph{multistep moves},
see Definition \ref{def:multistep_move}. These are sequences of transitions,
each allowed for the dynamics, leading to some desired final configuration.

\subsection{Structure of the paper}

In Section \ref{sec:model} we set up some of the notation, and define
kinetically constrained lattice gases. We also introduce two examples
that will be referred to throughout the paper.

Is Section \ref{sec:multistep_move} we introduce the notion of a
multistep move and its basic properties. We then use this notion in
order to precisely define of a mobile cluster and noncooperative models.
Finally, we provide a slightly weaker characterization of noncooperative
models.

The two following sections discuss the relaxation time in two different
settings---Section \ref{sec:relaxation_reservoir} concerns with
systems connected to a reservoir, while in Section \ref{sec:relaxation_closed}
we analyze closed systems. The result of Section \ref{sec:relaxation_reservoir}
shows diffusivity of the relaxation time in all noncooperative models.
It generalizes \cite{BT04}, and the proof uses the same strategy
in a wider context and in the language of multistep moves.

Studying the relaxation time in closed systems is much more involved.
This problem was analyzed for one noncooperative model in \cite{GoncalvesLandimToninelli},
proving diffusive scaling if the density is low enough or when adding
a small perturbation violating the constraint. The same model was
later considered in \cite{Nagahata12}, where diffusivity was proven
for all densities and with no perturbation. Here, in Section \ref{sec:relaxation_closed},
we generalize the result of \cite{Nagahata12} to some class of noncooperative
models. The proof of the result uses a completely different strategy---while
\cite{Nagahata12} relies on specific combinatorial details of the
model they study, the proof here only uses general properties of mobile
clusters. This new strategy allows us to obtain a result in a wider
context.

In Section \ref{sec:diffusion} we show that the diffusion coefficient
is positive for all noncooperative models. In order to achieve that,
we introduce a new comparison argument using multistep moves (Lemma
\ref{lem:D_comparison}). We then construct an auxiliary dynamics
which on one hand can be compared to the kinetically constrained gas
in question, and on the other hand possesses a special property allowing
us to calculate its diffusion coefficient explicitly.

The positivity of the self-diffusion coefficient for all noncooperative
models (in dimension $2$ and above) is proven in Section \ref{sec:self_diffusion}.
The proof applies a strategy similar to \cite[II.6]{Spohn2012IPS},
using a multistep move in order to compare the kinetically constrained
lattice gas to a random walk.

We conclude with open problems that this work suggests.

\section{\label{sec:model}The model}

\subsection{Notation}

In order to simplify the exposition of the model, we start by defining
some of the notation we use.
\begin{itemize}[leftmargin=*]
\item For $n\in\N$, we denote $[n]=\{1,\dots,n\}$.
\item We will consider models defined either on $\Z^{d}$, a finite box
$[L]^{d}$ for $L\in\N$, or the torus $\Z^{d}/L\Z^{d}$. We denote
by $\{e_{\alpha}\}_{\alpha=1}^{d}$ the standard basis, and we say
that two sites $x$ and $y$ are neighbors, denoted $x\sim y$, if
$x-y\in\{\pm e_{1},\dots,\pm e_{d}\}$. The \emph{boundary} of a set
$\Lambda\subset\Z^{d}$, denoted $\partial\Lambda$, is the set of
sites in $\Lambda$ that have a neighbor outside $\Lambda$.
\item For any finite sequence of sites $x_{1},\dots,x_{n}$, we denote by
$\sigma=(x_{1},\dots,x_{n})$ the corresponding cyclic permutation,
i.e., for any site $y$
\[
\sigma(y)=\begin{cases}
x_{k+1} & \text{if }y=x_{k}\text{ for }k\in[n-1],\\
x_{1} & \text{if }y=x_{n},\\
y & \text{otherwise}.
\end{cases}
\]
For a fixed site $x$ we denote by $\tau_{x}$ the permutation on
$\Z^{d}$ given by a translation by $x$, i.e., for any site $y\in\Z^{d}$
\[
\tau_{x}(y)=y+x.
\]
\item A configuration is an element $\eta$ of $\Omega=\Omega_{\Lambda}=\{0,1\}^{\Lambda}$,
where $\Lambda$ is either $\Z^{d}$, $[L]^{d}$, or the torus. We say that a
site $x\in\Lambda$ is \emph{empty} if $\eta(x)=0$ and \emph{occupied}
if $\eta(x)=1$.
\item For $\eta\in\Omega$ and a site $x$ we define $\eta^{x}$ to be the
configuration $\eta$ after flipping the occupation at $x$.
\item For $\eta\in\Omega$ and two sites $x$ and $y$ we define $\eta^{x,y}$
to be the configuration $\eta$ after exchanging the occupation values
at $x$ and $y$.
\item For $\eta\in\Omega$ and a permutation $\sigma$, we define $\sigma\eta$
to be the configuration after applying $\sigma$, i.e., for any site
$y$
\[
(\sigma\eta)(y)=\eta(\sigma^{-1}(y)).
\]
In particular, for any two sites $x$ and $y$ we can write $\eta^{x,y}=(x,y)\eta$.
\item For a $f:\Omega\to\R$ and two sites $x$ and $y$,
\begin{align*}
\grad_{x}f(\eta) & =f(\eta^{x})-f(\eta),\\
\grad_{x,y}f(\eta) & =f(\eta^{x,y})-f(\eta).
\end{align*}
\item For a $f:\Omega\to\R$ and a permutation $\sigma$, we define the
function $\sigma f$ as 
\[
\sigma f(\eta)=f(\sigma^{-1}\eta).
\]
\end{itemize}
Finally, we note that throughout the paper $C$ represents a generic
positive constant, that may depend only on the model (dimension and
constraints), and in particular does \emph{not} depend on the parameter
$q$.

\subsection{Kinetically constrained lattice gases}

Kinetically constrained lattice gases are interacting particle systems,
defined on $\Z^{d}$, with generator $\cL$ acting on any local function
$f:\Omega\to\R$ as

\begin{equation}
\cL f(\eta)=\sum_{x\sim y}c_{x,y}(\eta)\grad_{x,y}f(\eta).\label{eq:generator}
\end{equation}
The rates $c_{x,y}$ must have the following properties:
\begin{enumerate}
\item For any $x\sim y$ and $\eta\in\Omega$, $c_{x,y}(\eta)\in\{0\}\cup[1,c_{\text{max}}]$
for some $c_{\text{max}}\ge1$.
\item The rate $c_{x,y}$ depends only on the configuration outside $x$
and $y$.
\item The rates are nondegenerate, i.e., for any edge $x\sim y$ there exists
a configuration $\eta\in\Omega$ such that $c_{x,y}(\eta)\ge1$ and
a configuration $\eta'\in\Omega$ such that $c_{x,y}(\eta)=0$.
\item For fixed $x$ and $y$, the rate is a decreasing function of $\eta$,
i.e., emptying sites could only speed up the dynamics.
\item The model is homogeneous: $c_{x,y}(\eta)=c_{\tau_{z}(x),\tau_{z}(y)}(\tau_{z}\eta)$
for any $\eta\in\Omega$ and $x,y,z\in\Z^d$.
\item The rates have finite range, i.e., $c_{x,y}$ depends only on the
occupation of the sites in some box $x+[-R,R]$, where $R$ is called
the \emph{range}.
\end{enumerate}
Sometimes we refer to the rate $c_{x,y}$ as the \emph{constraint}
(having in mind the case $c_{\text{max}}=1$), and say that the constraint
is satisfied when $c_{x,y}\ge1$ and not satisfied when $c_{x,y}=0$.

We may also consider the model on a subset of the lattice $\Lambda\subset\Z^{d}$
(usually $[L]^{d}$) by thinking of the sites outside $\Lambda$ as
empty. The generator has the same form as (\ref{eq:generator}), with
sum taken over $x,y\in\Lambda$. The constraint $c_{x,y}(\eta)$ for
$\eta\in\{0,1\}^{\Lambda}$ is then defined to be $c_{x,y}(\overline{\eta})$,
where $\overline{\eta}\in\{0,1\}^{\Z^{d}}$ is the configuration which
equals $\eta$ on $\Lambda$ and $0$ outside $\Lambda$. Theses are the 
\emph{empty boundary conditions}. The \emph{occupied boundary conditions}
are defined analogously. Finally, \emph{periodic boundary conditions} are 
defined when considering the model on the torus. The constraint $c_{x,y}(\eta)$ 
for $\eta\in\{0,1\}^{\Z^{d}/L\Z^{d}}$ is then given by $c_{x,y}(\overline{\eta})$ 
with $\overline{\eta}(x)=\eta(x\text{ mod }L^{d})$.

Under the assumptions above, the dynamics is reversible with respect
to a product measure for any density in $[0,1]$. We refer to this
measure as the \emph{equilibrium measure} (at a given density). The
density of \emph{empty} sites is denoted by $q\in[0,1]$, so the equilibrium
measure $\mu=\mu_{q}$ assigns to each site an independent Bernoulli
random variable with parameter $1-q$.

On a finite box $\Lambda=[L]^{d}$, we may consider a kinetically
constrained lattice gas with reservoir on the boundary. This model
is defined by the generator $\cL_{\text{r}}$ operating on any local
function $f:\Omega\to\R$ as

\begin{equation}
\cL_{\text{r}}f(\eta)=\sum_{\substack{x,y\in\Lambda\\
x\sim y
}
}c_{x,y}(\eta)\grad_{x,y}f(\eta)+\sum_{x\in\partial\Lambda}c_{x}\grad_{x}f(\eta),\label{eq:generator_with_reservoir}
\end{equation}
where $c_{x}(\eta)=q\eta(x)+(1-q)(1-\eta(x))$. Note that $c_{x}(\eta)$
is chosen such that the process remains reversible with respect to
$\mu$.

\subsection{Examples}

Throughout the paper, we will refer to two fundamental examples:
\begin{example}
\label{exa:1d}The $1$ dimensional model, with constraint 
\[
c_{x,x+1}(\eta)=\begin{cases}
1 & \text{if }\eta(x-1)=0\text{ or }\eta(x+2)=0,\\
0 & \text{otherwise}.
\end{cases}
\]

This model was introduced in \cite{BT04}, and further studied in
\cite{Nagahata12}. In \cite{GoncalvesLandimToninelli} a slight variation
was introduced, where the rate $c_{x,x+1}$ equals $2$
if both $\eta(x-1)$ and $\eta(x+2)$ are empty. This difference is
of no importance to the analysis in this paper, but it does introduce
a significant simplification in proving the convergence to a hydrodynamic
limit.
\end{example}

\begin{example}
\label{exa:2d}The $2$ dimensional model with constraint 
\[
c_{x,x+e_{\alpha}}(\eta)=\begin{cases}
1 & \text{if }\eta(x-e_{\alpha})=0\text{ or }\eta(x+2e_{\alpha})=0,\\
0 & \text{otherwise},
\end{cases}
\]
for $\alpha\in\{1,2\}$. 

This could be seen is a generalization of Example \ref{exa:1d}, also
studied in \cite{BT04}.
\end{example}

\section{\label{sec:multistep_move}Multistep moves}

The main tool we use in this paper are \emph{multistep moves}, which
are sequences of transitions allowed for the dynamics, taking us from
one configuration to the other. This formulation, used in \cite{MST19KA,ES20KAtagged,S23KA_HL},
makes the application of canonical path methods more transparent.

A multistep move provides, for $\eta$ in some fixed set of configuration (the domain), a sequence
of transitions that are allowed for the dynamics. That is, at each step $t$ it will tell
us which edge to exchange in order to move from the configuration $\eta_t$ to $\eta_{t+1}$.
In order for the move to be valid, in all exchanges the constraint must be satisfied.
This is expressed in the following definition:
\begin{defn}
[Multistep move]\label{def:multistep_move} For fixed $T>0$, a $T$-step
move $M$ defined on $\Dom M\subseteq\Omega$ is a triple $\left((\eta_{t})_{t=0}^{T},(x_{t})_{t=0}^{T-1},(e_{t})_{t=0}^{T-1}\right)$;
where $(\eta_{t})_{t=0}^{T}$ is a sequence of functions $\eta_{t}:\Dom M\to\Omega$, 
$(x_{t})_{t=0}^{T-1}$ is a sequence of functions
$x_{t}:\Dom M\to\Z^{d}$, and $(e_{t})_{t=0}^{T-1}$ is a sequence
of functions $e_{t}:\Dom M\to\left\{ \pm e_{1},\dots,\pm e_{d}\right\} $.
The move must satisfy the following properties:
\begin{enumerate}
\item For any $\eta\in\Dom M$, $\eta_{0}(\eta)=\eta$.
\item For any $\eta\in\Dom M$ and $t\in\{0,\dots,T-1\}$,
\begin{enumerate}
\item on the infinite lattice or a finite box with no reservoirs, 
\[
\eta_{t+1}(\eta)=\eta_{t}(\eta)^{x_{t}(\eta),x_{t}(\eta)+e_{t}(\eta)}\text{ and }c_{x_{t}(\eta),x_{t}(\eta)+e_{t}(\eta)}(\eta_{t}(\eta))=1.
\]
\item on a finite box $\Lambda$ with reservoirs, either 
\[
\eta_{t+1}(\eta)=\eta_{t}(\eta)^{x_{t}(\eta),x_{t}(\eta)+e_{t}(\eta)}\text{ and }c_{x_{t}(\eta),x_{t}(\eta)+e_{t}(\eta)}(\eta_{t}(\eta))=1,
\]
 or 
\[
\eta_{t+1}(\eta)=\eta_{t}(\eta)^{x_{t}(\eta)}\text{ and }x_{t}(\eta)\in\partial\Lambda.
\]
\end{enumerate}
\end{enumerate}
When context allows we omit, with some abuse of notation, the explicit
dependence on $\eta$ (writing $\eta_{t},x_{t},e_{t}$ rather than
$\eta_{t}(\eta),x_{t}(\eta),e_{t}(\eta)$).
\end{defn}

We continue with several basic notions related to multistep moves.
\begin{defn}
[Information loss]Consider a $T$-step move $M=\left((\eta_{t}),(x_{t}),(e_{t})\right)$
and $t\in\{0,\dots,T\}$. The \emph{loss of information at time $t$
}is defined as 
\[
2^{\Loss_{t}M}=\sup_{\eta',x',e'}\#\left\{ \eta\in\Dom M\text{ such that }\eta_{t}(\eta)=\eta'\text{, }x_{t}(\eta)=x'\text{ and }e_{t}(\eta)=e'\right\} ,
\]
where the supremum is taken over $\eta'\in\Dom M$, $x'\in\Z^{d}$
and $e'\in\{\pm e_{1},\dots,\pm e_{d}\}$. We also define 
\[
\Loss M=\sup_{t}\Loss_{t}M.
\]
That is, for given $t,\eta',x',e'$ there are at most $2^{\Loss M}$
possible configurations $\eta\in\Dom M$ for which $\eta_{t}=\eta'$,
$x_{t}=x'$ and $e_{t}=e'$.
\end{defn}

\begin{defn}
[Energy barrier]Consider a $T$-step move $M=\left((\eta_{t}),(x_{t}),(e_{t})\right)$
for a kinetically constrained lattice gas defined on a finite box
$\Lambda$ with reservoirs on the boundaries. The energy barrier is
\[
\EBarrier(M)=\sup_{t\in\{0,\dots T\}}\sup_{\eta\in\Dom\Omega}\left(\#\left\{ \text{empty sites in }\eta_{t}\right\} -\#\left\{ \text{empty sites in }\eta\right\} \right).
\]
Note that, since $\eta_{0}=\eta$, $\EBarrier(M)\ge0$.
\end{defn}

\begin{defn}
[Composition of multistep moves] Fix a $T_{1}$-step move $M_{1}=\left((\eta_{t}^{1}),(x_{t}^{1}),(e_{t}^{1})\right)$
and a $T_{2}$-step move $M_{2}=\left((\eta_{t}^{2}),(x_{t}^{2}),(e_{t}^{2})\right)$
such that for any $\eta\in\Dom M_{1}$, $\eta_{T_{1}}^{1}(\eta)\in\Dom M_{2}$.
Then their composition $M_{2}\circ M_{1}$ is the $T$-step move $M=\left((\eta_{t}),(x_{t}),(e_{t})\right)$,
with $T=T_{1}+T_{2}$ and $\Dom M=\Dom M_{1}$ given by
\begin{align*}
\eta_{t}(\eta) & =\begin{cases}
\eta_{t}^{1}(\eta) & \text{if }t\in\{0,\dots,T_{1}\},\\
\eta_{t-T_{1}}^{2}(\eta_{T}^{1}(\eta)) & \text{otherwise},
\end{cases}\\
x_{t}(\eta) & =\begin{cases}
x_{t}^{1}(\eta) & \text{if }t\in\{0,\dots,T_{1}\},\\
x_{t-T_{1}}^{2}(\eta_{T}^{1}(\eta)) & \text{otherwise},
\end{cases}\\
e_{t}(\eta) & =\begin{cases}
e_{t}^{1}(\eta) & \text{if }t\in\{0,\dots,T_{1}\},\\
e_{t-T_{1}}^{2}(\eta_{T}^{1}(\eta)) & \text{otherwise}.
\end{cases}
\end{align*}
\end{defn}

\begin{defn}
[Associated permutation]We consider here a model with no reservoirs.
Fix a $T$-step move $M=\left((\eta_{t}),(x_{t}),(e_{t})\right)$
and $\eta\in\Dom M$. Then the \emph{associated permutation} $\sigma$
is a permutation on the sites of $\Z^{d}$ given by the product of transpositions 
$(x_{T-1},x_{T-1}+e_{T-1})(x_{T-2},x_{T-2}+e_{T-2})\dots(x_{0},x_{0}+e_{0})$.

We say that the move $M$ is \emph{compatible with a permutation}
$\sigma$ if, for any $\eta\in\Dom M$, the associated permutation
is $\sigma$.
\end{defn}

\begin{observation}
Fix a $T$-step move $M=\left((\eta_{t}),(x_{t}),(e_{t})\right)$
and $\eta\in\Dom M$. Then $\eta_{T}=\sigma\eta$, i.e., for any $x\in\Z^{d}$,
\[
\eta_{T}(\sigma(x))=\eta(x).
\]
\end{observation}

\begin{observation}
Consider two multistep moves $M_{1}$ and $M_{2}$. Assume that $M_{1}$
is compatible with a permutation $\sigma_{1}$ and $M_{2}$ with a
permutation $\sigma_{2}$. If $M_{2}\circ M_{1}$ is well defined,
then it is compatible with $\sigma_{2}\sigma_{1}$.
\end{observation}

\begin{defn}
[Deterministic move]A $T$-step move $M=\left((\eta_{t}),(x_{t}),(e_{t})\right)$
is called \emph{deterministic} if the sequences $(x_{t})_{t=0}^{T-1}$
and $(e_{t})_{t=0}^{T-1}$ do not depend on $\eta$, that is, for
any $\eta,\eta'\in\Dom M$ and any $t\in\{0,\dots,T-1\}$, $x_{t}(\eta)=x_{t}(\eta')$
and $e_{t}(\eta)=e_{t}(\eta')$. Note that a deterministic move is
always compatible with a permutation, and has $0$ loss of information.
\end{defn}

\begin{observation}
Consider a deterministic $T$-step move $M=\left((\eta_{t}),(x_{t}),(e_{t})\right)$
compatible with a permutation $\sigma$. The there exists an inverse
move $M^{-1}$ with domain 
\[
\Dom M^{-1}=\left\{ \eta\in\Omega:\sigma\eta\in\Dom M\right\} ,
\]
which is a $T$-step move compatible with $\sigma^{-1}$.
\end{observation}

These are the general definitions and basic properties of multistep
moves. We now continue with a few definitions related to the noncooperative
nature of the model. In each definition, we will describe a move
that changes the configuration in a desired way without ``disturbing''
too many sites, under the condition that there is a mobile cluster
near by. The way in which we change the configuration is given by
the permutation the move is compatible with. The fact that we do not
want to disturb many sites is expressed in the fact that all $x_{t}$'s
are restricted to some given box. The requirement that a mobile cluster
is available is expressed in the domain of the multistep move.

The first move we define will allow us to move a mobile cluster $\MC$
on the lattice:
\begin{defn}
[Translation move]\label{def:translation_move}Fix a finite set $\MC\subset\Z^{d}$,
$\MClen>0$, $e\in\{\pm e_{1},\dots,\pm e_{d}\}$ and $x\in\Z^{d}$.
A \emph{translation move} in $[-\MClen,\MClen]^{d}$ of the cluster $x+\MC$
in the direction $e$ is a $T_{\text{Tr}}$-step move $\text{Tr}_{e}(x+\MC)$
satisfying:
\begin{enumerate}
\item $\Dom\text{Tr}_{e}(x+\MC)=\left\{ \eta\in\Omega:x+\MC\text{ is empty}\right\} $
\item $\text{Tr}_{e}(x+\MC)$ is a deterministic move, compatible with a
permutation $\sigma$.
\item $\sigma(x+y)=x+y+e$ for any $y\in\MC$.
\item For all $t\in\{0,\dots,T-1\}$, $x_{t}\in x+[-\MClen,\MClen]^{d}$
and $x_{t}+e_{t}\in x+[-\MClen,\MClen]^{d}$.
\end{enumerate}
For brevity, we may write $\text{Tr}_{\pm\alpha}$ rather than $\text{Tr}_{\pm e_{\alpha}}$.
\end{defn}

\begin{observation}
Fix a $\MC\subset\Z^{d}$, $\MClen>0$, $e\in\{\pm e_{1},\dots,\pm e_{d}\}$
and $x\in\Z^{d}$. Then $\text{Tr}_{e}(x+\MC)^{-1}$ is a translation
move in $[-\MClen,\MClen]^{d}$ of the cluster $x+e+\MC$ in the direction
$-e$. We may therefore always assume that the translation moves are
chosen such that $\text{Tr}_{e}(x+\MC)^{-1}=\text{Tr}_{-e}(x+e+\MC)$.
\end{observation}

Once we are able to move the mobile cluster around, we need to use
it in order to move particles in its vicinity.
\begin{defn}
[Exchange move]\label{def:exchange_move}Fix a finite set $\MC\subset\Z^{d}$,
$\MClen>0$, $e\in\{\pm e_{1},\dots,\pm e_{d}\}$ and $x\in\Z^{d}$.
An \emph{exchange move} in $[-\MClen,\MClen]^{d}$ using the cluster $x+\MC$
in the direction $e$ is a $T_{\text{Ex}}$-step move $\text{Ex}_{e}(x+\MC)$
satisfying:
\begin{enumerate}
\item $\Dom\text{Ex}_{e}(x+\MC)=\left\{ \eta\in\Omega:x+\MC\text{ is empty}\right\} $
\item $\text{Ex}_{e}(x+\MC)$ is a deterministic move, compatible with the
permutation $(x+y,x+y+e)$, where $y=\MClen e$.
\item For all $t\in\{0,\dots,T-1\}$, $x_{t}\in x+[-\MClen,\MClen]^{d}\cup\{x+(\MClen+1)e\}$
and $x_{t}+e_{t}\in x+[-\MClen,\MClen]^{d}\cup\{x+(\MClen+1)e\}$.
\end{enumerate}
\end{defn}

\begin{defn}
[Mobile cluster]\label{def:mobile_cluster_noncooperative}A \emph{mobile
cluster} $\MC$ is a finite set of sites, for which there exists $\MClen>0$
such that $\text{Tr}_{e}(x+\MC)$ and $\text{Ex}_{e}(x+\MC)$ could
be constructed for all $e$ and $x$. Equivalently, by translation invariance,
there exists $\MClen>0$ such that $\text{Tr}_{e}(\MC)$ and $\text{Ex}_{e}(\MC)$
could be constructed for all $e$.

A kinetically constrained lattice gas is called \emph{noncooperative}
if there exists a mobile cluster.
\end{defn}

\begin{example}
\label{exa:tr_1d}The model in Example \ref{exa:1d} is noncooperative---take
$\MC=\{1,2\}$ and $l=3$. We need to construct four moves: $\text{Tr}_{1}(\MC),\text{Tr}_{-1}(\MC),\text{Ex}_{1}(\MC),\text{Ex}_{-1}(\MC)$.

$\text{Tr}_{1}(\MC)$ will be a $2$-step move $\left((\eta_{0},\eta_{1},\eta_{2}),(x_{0},x_{1}),(e_{0},e_{1})\right)$
operating on $\eta\in\Dom\text{Tr}_{1}(\MC)$ as follows:
\begin{alignat*}{1}
\eta_{0} & =\eta,\quad\eta_{1}=\eta^{2,3}=(2,3)\eta,\quad\eta_{2}=(2,3,1)\eta,\\
x_{0} & =2,e_{0}=1,\quad x_{1}=1,e_{1}=1.
\end{alignat*}
Recalling that $\eta\in\Dom\text{Tr}_{1}(\MC)$ means $\eta(1)=\eta(2)=0$,
it is straightforward to verify that the move is well defined and
that it is indeed a translation move. See Figure \ref{fig:mc_1d}.

$\text{Tr}_{-1}(\MC)$ is defined as $\text{Tr}_{1}(-1+\MC)^{-1}$.

$\text{Ex}_{1}(\MC)$ is the $1$-step move exchanging the sites $3$
and $4$, which is allowed since $2$ must be empty.

$\text{Ex}_{-1}(\MC)$ could be constructed as the composition 
\begin{multline*}
\text{Ex}_{-1}(\MC)=\text{Tr}_{-1}(-1+\MC)^{-1}\circ\text{Tr}_{-1}(-2+\MC)^{-1}\circ\text{Tr}_{-1}(-3+\MC)^{-1}\circ\text{Tr}_{-1}(-4+\MC)^{-1}\circ\text{Ex}_{1}(-5+\MC)\\
\circ\text{Tr}_{-1}(-4+\MC)\circ\text{Tr}_{-1}(-3+\MC)\circ\text{Tr}_{-1}(-2+\MC)\circ\text{Tr}_{-1}(-1+\MC)\circ\text{Tr}_{-1}(\MC).
\end{multline*}
The composition is well defined (recalling $\text{Tr}_{-1}(x+\MC)^{-1}=\text{Tr}_{1}(x-1+\MC)$,
so its domain consists of the configurations where $x-1+\MC$ is empty).
Moreover, it is a composition of deterministic moves, and compatible
with 
\begin{multline*}
(1,2,0)(0,1,-1)(-1,0,-2)(-2,-1,-3)(-3,-2,-4)(-2,-1)\\
(-4,-2,-3)(-3,-1,-2)(-2,0,-1)(-1,1,0)(0,2,1)=(-3,-4).
\end{multline*}
\end{example}

\begin{example}
\label{exa:tr_2d}The model in Example \ref{exa:2d} is noncooperative, with
$\MC=\{e_{1}+e_{2},e_{1}+2e_{2},2e_{1}+e_{2},2e_{1}+2e_{2}\}$ and
$\MClen=3$. The construction of the multistep moves is the same as
the previous example, see Figure \ref{fig:mc_2d}.

\begin{figure}
\begin{tikzpicture}[scale=0.3, every node/.style={scale=0.6}]
	\draw[step=1,gray] (0,0) grid +(5,5);
	\foreach \x in {0,3,4}{
		\foreach \y in {0,...,4}{
			\fill[black,shift={(0.5,0.5)}] (\x,\y) circle (0.4);
		}
	}
	\foreach \x in {1,2}{
		\foreach \y in {0,3,4}{
			\fill[black,shift={(0.5,0.5)}] (\x,\y) circle (0.4);
		}
	}
	
	\draw[->,shift={(0.5,0.5)}]  (6,2) to (8,2);

	\draw[step=1,gray] (10,0) grid +(5,5);
	\foreach \x in {10,13,14}{
		\foreach \y in {0,...,4}{
			\fill[black,shift={(0.5,0.5)}] (\x,\y) circle (0.4);
		}
	}
	\foreach \x in {11,12}{
		\foreach \y in {0,4}{
			\fill[black,shift={(0.5,0.5)}] (\x,\y) circle (0.4);
		}
	}
	\fill[black,shift={(0.5,0.5)}] (11,2) circle (0.4);
	\fill[black,shift={(0.5,0.5)}] (12,3) circle (0.4);
	
	\draw[->,shift={(0.5,0.5)}]  (16,2) to (18,2);
	
	\draw[step=1,gray] (20,0) grid +(5,5);
	\foreach \x in {20,23,24}{
		\foreach \y in {0,...,4}{
			\fill[black,shift={(0.5,0.5)}] (\x,\y) circle (0.4);
		}
	}
	\foreach \x in {21,22}{
		\foreach \y in {0,4}{
			\fill[black,shift={(0.5,0.5)}] (\x,\y) circle (0.4);
		}
	}
	\fill[black,shift={(0.5,0.5)}] (21,1) circle (0.4);
	\fill[black,shift={(0.5,0.5)}] (22,3) circle (0.4);
	
	\draw[->,shift={(0.5,0.5)}]  (26,2) to (28,2);
	
	\draw[step=1,gray] (30,0) grid +(5,5);
	\foreach \x in {30,33,34}{
		\foreach \y in {0,...,4}{
			\fill[black,shift={(0.5,0.5)}] (\x,\y) circle (0.4);
		}
	}
	\foreach \x in {31,32}{
		\foreach \y in {0,4}{
			\fill[black,shift={(0.5,0.5)}] (\x,\y) circle (0.4);
		}
	}
	\fill[black,shift={(0.5,0.5)}] (31,1) circle (0.4);
	\fill[black,shift={(0.5,0.5)}] (32,2) circle (0.4);
	
	\draw[->,shift={(0.5,0.5)}]  (36,2) to (38,2);
		
	\draw[step=1,gray] (40,0) grid +(5,5);
	\foreach \x in {40,43,44}{
		\foreach \y in {0,...,4}{
			\fill[black,shift={(0.5,0.5)}] (\x,\y) circle (0.4);
		}
	}
	\foreach \x in {41,42}{
		\foreach \y in {0,4}{
			\fill[black,shift={(0.5,0.5)}] (\x,\y) circle (0.4);
		}
	}
	\fill[black,shift={(0.5,0.5)}] (41,1) circle (0.4);
	\fill[black,shift={(0.5,0.5)}] (42,1) circle (0.4);
\end{tikzpicture}

\caption{\label{fig:mc_2d}This is an illustration of the translation move
in the model described in examples \ref{exa:2d} and \ref{exa:tr_2d}.
The mobile cluster is given by an empty $2\times2$ square. In this
figure we see how it could move one step up.}

\end{figure}
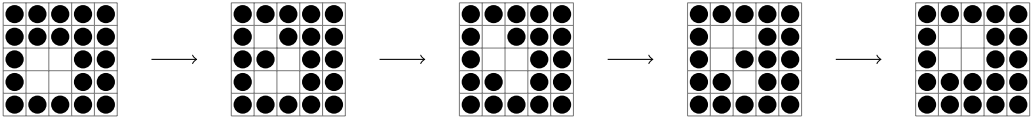
\end{example}

To conclude this section, we see in the following proposition that
if we are able to construct, for any direction, a cluster that is
free to move in that direction, then the model is noncooperative,
i.e., there is some (possible very large) cluster that is able to
move in all directions, and to exchange edges in its vicinity.
\begin{prop}
\label{prop:exchange_move}Assume that for any $e\in\{e_{1},\dots,e_{d}\}$
there exists $\MC_{e}$ and $\MClen_{e}$, such that $\text{Tr}_{e}(\MC_{e})$
exists. Then the model is noncooperative, i.e., there exists
a mobile cluster $\MC$.
\end{prop}

\begin{proof}
The construction of the cluster $\MC$ is explained in the appendix
of \cite{S23KA_HL} (claims A11 and on). Since the result there is
stated in a slightly different context (and with different notation),
we explain here briefly how the cluster is constructed. The reader
may consult \cite{S23KA_HL} for any missing details.
\begin{claim}
\label{claim:ex_construction}Fix $e$. If $\text{Tr}_{e}(\MC)$ exists
for some $\MC$ and $\MClen$, then $\text{Tr}_{e}(\MC')$ and $\text{Ex}_{e}(\MC')$
exist for some $\MC'$ and $\MClen'$.
\end{claim}

\begin{proof}
Without loss of generality $e=e_{1}$. Let $\{y_{1},\dots,y_{k}\}\in(\infty,0]\times\Z^{d-1}$
be finite set of sites such that $c_{0,e}\ge1$ if $\{y_{1},\dots,y_{k}\}$
is empty. This set has to exist since $\text{Tr}_{e}(\MC)$ exists.
Fix $\MC'=\bigcup_{i=1}^{k}\left(y_{i}-i\MClen e_{1}+\MC\right)$.
Define $\text{Ex}(\MC')$ by translating the copies of $\MC$ until
$y_{1},\dots,y_{k}$ are all empty, then exchange $0$ and $e$, and
finally roll back the translation moves. See Figure \ref{fig:ex_construction}.
\end{proof}
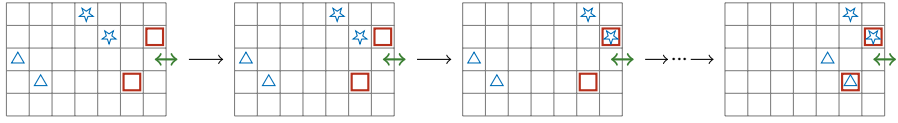
\begin{figure}
\begin{tikzpicture}[scale=0.3, every node/.style={scale=0.6}]
	\draw[step=1,gray] (0,0) grid +(7,5);
	\draw[<->,OliveGreen,thick] (6.5,2.5) to (7.5,2.5);
	
	\node[rectangle, BrickRed, draw, scale=1.3, thick] at (5.5,1.5) {};
	\node[rectangle, BrickRed, draw, scale=1.3, thick] at (6.5,3.5) {};
	
	\node[star, star point ratio=0.4, draw, NavyBlue, scale=0.8] at (4.5,3.5) {};
	\node[star, star point ratio=0.4, draw, NavyBlue, scale=0.8] at (3.5,4.5) {};	
	
	\node[regular polygon,regular polygon sides=3, scale=0.4, draw, NavyBlue] at (1.5,1.5) {};
	\node[regular polygon,regular polygon sides=3, scale=0.4, draw, NavyBlue] at (0.5,2.5) {};

	\draw[->]  (8,2.5) to (9.5,2.5);

	\def \x{10}
	\draw[step=1,gray,shift={(\x,0)}] (0,0) grid +(7,5);
	\draw[<->,OliveGreen,thick,shift={(\x,0)}] (6.5,2.5) to (7.5,2.5);
	
	\draw[shift={(\x,0)}] (5.5,1.5) node[rectangle, BrickRed, scale=1.3, thick, draw] {};	
	\draw[shift={(\x,0)}] (6.5,3.5) node[rectangle, BrickRed, scale=1.3, thick, draw] {};	
	
	\draw[shift={(\x,0)}] (5.5,3.5) node[star, star point ratio=0.4, NavyBlue, scale=0.8, draw] {};
	\draw[shift={(\x,0)}] (4.5,4.5) node[star, star point ratio=0.4, NavyBlue, scale=0.8, draw] {};	
	
	\draw[shift={(\x,0)}] (1.5,1.5) node[regular polygon,regular polygon sides=3, scale=0.4, draw, NavyBlue] {};
	\draw[shift={(\x,0)}] (0.5,2.5) node[regular polygon,regular polygon sides=3, scale=0.4, draw, NavyBlue] {};

	\draw[->]  (18,2.5) to (19.5,2.5);

	\def \x{20}
	\draw[step=1,gray,shift={(\x,0)}] (0,0) grid +(7,5);
	\draw[<->,OliveGreen,thick,shift={(\x,0)}] (6.5,2.5) to (7.5,2.5);
	
	\draw[shift={(\x,0)}] (5.5,1.5) node[rectangle, BrickRed, scale=1.3, thick, draw] {};	
	\draw[shift={(\x,0)}] (6.5,3.5) node[rectangle, BrickRed, scale=1.3, thick, draw] {};	
	
	\draw[shift={(\x,0)}] (6.5,3.5) node[star, star point ratio=0.4, NavyBlue, scale=0.8, draw] {};
	\draw[shift={(\x,0)}] (5.5,4.5) node[star, star point ratio=0.4, NavyBlue, scale=0.8, draw] {};	
	
	\draw[shift={(\x,0)}] (1.5,1.5) node[regular polygon,regular polygon sides=3, scale=0.4, draw, NavyBlue] {};
	\draw[shift={(\x,0)}] (0.5,2.5) node[regular polygon,regular polygon sides=3, scale=0.4, draw, NavyBlue] {};

	\draw[->]  (28,2.5) to (29,2.5);
	\draw  (29.5,2.5) node {...};
	\draw[->]  (30,2.5) to (31,2.5);

	\def \x{31.5}
	\draw[step=1,gray,shift={(\x,0)}] (0,0) grid +(7,5);
	\draw[<->,OliveGreen,thick,shift={(\x,0)}] (6.5,2.5) to (7.5,2.5);
	
	\draw[shift={(\x,0)}] (5.5,1.5) node[rectangle, BrickRed, scale=1.3, thick, draw] {};	
	\draw[shift={(\x,0)}] (6.5,3.5) node[rectangle, BrickRed, scale=1.3, thick, draw] {};	
	
	\draw[shift={(\x,0)}] (6.5,3.5) node[star, star point ratio=0.4, NavyBlue, scale=0.8, draw] {};
	\draw[shift={(\x,0)}] (5.5,4.5) node[star, star point ratio=0.4, NavyBlue, scale=0.8, draw] {};	
	
	\draw[shift={(\x,0)}] (5.5,1.5) node[regular polygon,regular polygon sides=3, scale=0.4, draw, NavyBlue] {};
	\draw[shift={(\x,0)}] (4.5,2.5) node[regular polygon,regular polygon sides=3, scale=0.4, draw, NavyBlue] {};	
\end{tikzpicture}

\caption{\label{fig:ex_construction}We see here how the exchange move could
be constructed, see Claim \ref{claim:ex_construction}. For the sake
of this illustration, we assume that it suffices to empty the two
sites marked with a red square in order to free the edge $(0,e_{1})$
marked in green. The mobile cluster $\protect\MC$, marked with blue
stars, is empty. In addition, a translation of $\protect\MC$, marked
with blue triangles, is also empty. After applying the multistep move
described in the figure the constraint is satisfied at the edge $(0,e_{1})$,
so we may exchange the two sites and move the mobile clusters back
to their original position.}

\end{figure}

\begin{claim}
Assume $\text{Tr}_{1}(\MC_{1})$, $\text{Ex}_{1}(\MC_{1})$, $\text{Tr}_{2}(\MC_{2})$,
$\text{Ex}_{2}(\MC_{2})$,...,$\text{Tr}_{k}(\MC_{k})$, $\text{Ex}_{k}(\MC_{k})$
are defined. Then there exist $\MC_{k}'$ and $\MClen_{k}'$ such
that for all $y\in[l_{1},\infty]e_{1}+\Z e_{2}+\dots+\Z e_{k}$ we
may define a multistep move $\text{Ex}_{1}^{y}$ exchanging $(y,y+e_{1})$.
\end{claim}

\begin{proof}
Consider first $k=2$, and denote $\MC_{1}=\{x_{1},\dots,x_{n}\}$.
We choose 
\[
\MC'_{2}=\MC_{1}\cup\left(\bigcup_{i=1}^{n}x_{i}-le_{2}+\MC_{2}\right).
\]
By applying translation and exchange moves using the cluster $x_{i}-le_{2}+\MC_{2}$,
we are able to exchange $x_{i}$ with $x_{i}+e_{2}$. Doing that for
all $i$, we end up with an empty cluster $(e_{2}+\MC_{1})\cup\left(\bigcup_{i=1}^{n}x_{i}-le_{2}+\MC_{2}\right)$.
We can repeat the operation (with one additional translation move
for each $i$), reaching an empty cluster $(2e_{2}+\MC_{1})\cup\left(\bigcup_{i=1}^{n}x_{i}-le_{2}+\MC_{2}\right)$.
In fact, by adjusting the number of repetitions we are able to empty
all sites of $(w+\MC_{1})\cup\left(\bigcup_{i=1}^{n}x_{i}-le_{2}+\MC_{2}\right)$
where $w=y-(y\cdot e_{1})e_{1}$. Now, since $w+\MC_{1}$ is empty,
we can use $\text{Tr}_{1}(w+\MC_{1})$ and $\text{Ex}_{1}(w+\MC_{1})$
in order to exchange $y$ and $y+e_{1}$. Rolling back all changes,
we end up with the move $\text{Ex}_{1}^{w}$.

For larger values of $k$ we follow the same construction by induction---use
$\left|\MC'_{k-1}\right|$ copies of $\MC_{k}$ in order to move a
single copy of $\MC'_{k-1}$ in the $e_{k}$ direction $y\cdot e_{k}$
times. Then apply (the translation of) $\text{Ex}_{1}^{y-y\cdot e_{k}}$
in order to exchange $y$ and $y+e_{1}$, and roll back to place $\MC'_{k}$
in its original location.
\end{proof}
This claim allows us to define a cluster $\MC'_{d}$, which allows
exchanges in the direction $e_{1}$. We may construct in the same
manner clusters allowing exchanges in any direction:
\begin{cor}
For any $e$, there exist $\overline{\MClen}_{e}$ and $\overline{\MC}_{e}$,
such that we may define a multistep move $\text{Ex}_{e}^{y}(x+\overline{\MC}_{e})$
exchanging $x+y$ with $x+y+e$ whenever $x+\overline{\MC}_{e}$ is
empty, for all $y$ such that $y\cdot e\ge\overline{\MClen}_{e}$.
\end{cor}

To conclude, consider $2d$ disjoint copies of the clusters defined
in the corollary above placed on the diagonal---
\[
\MC=\left(\bigcup_{\alpha=1}^{d}-\alpha \MClen(1,1,\dots,1)+\overline{\MC}_{e_{\alpha}}\right)\bigcup\left(\bigcup_{\alpha=1}^{d} \alpha \MClen (1,1,\dots,1)+\overline{\MC}_{e_{-\alpha}}\right),
\]
for large enough $\MClen$ to guarantee that the union is indeed disjoint.
Now, in order to construct $\text{Ex}_{e_{\alpha}}(\MC)$ we may simply
use $\text{Ex}_{e_{\alpha}}^{y}(x+\overline{\MC}_{e_{\alpha}})$ with
$x=-\alpha\MClen(1,\dots,1)$ and $y=\alpha \MClen e_{\alpha}-x$ (and analogously for
$\text{Ex}_{e_{-\alpha}}(\MC)$). In order to construct $\text{Tr}_{e_{\alpha}}$,
we first use the cluster $-\alpha \MClen (1,\dots,1)+\overline{\MC}_{e_{\alpha}}$
in order to move in the direction $e_{\alpha}$ all vacancies in $\bigcup_{\alpha=1}^{d}\left(\alpha \MClen (1,1,\dots,1)+\MC_{e_{-\alpha}}\right)$.
Then we use the cluster $\alpha \MClen (1,\dots,1)+e_{\alpha}+\overline{\MC}_{e_{-\alpha}}$
in order to move all vacancies in $\bigcup_{\alpha=1}^{d}\left(-\alpha \MClen (1,1,\dots,1)+\overline{\MC}_{e_{\alpha}}\right)$
in the direction $e_{\alpha}$. This concludes the proof of the proposition. 
\end{proof}

\section{\label{sec:relaxation_reservoir}Relaxation time on a finite box
with a reservoir}

In this section we consider noncooperative kinetically constrained
lattice gases on a finite box $[L]^{d}$ with reservoirs on the boundary.
In \cite{BT04}, the relaxation times of two models were studied,
and a diffusive scaling was proven. We will follow their strategy,
showing a diffusive scaling with power law dependence on $q$.

In order to define the relaxation time, we first write the Dirichlet
form associated with the generator $\cL_{\text{r}}$ given in equation
(\ref{eq:generator_with_reservoir}):
\begin{equation}
\cD_{\text{r}}f=\mu\left[\sum_{x\sim y\in\Lambda}c_{x,y}(\grad_{x,y}f)^{2}\right]+\mu\left[\sum_{x\in\partial\Lambda}c_{x}(\grad_{x}f)^{2}\right].\label{eq:dirichlet_with_reservoir}
\end{equation}
Then the relaxation time is given by 
\begin{equation}
\sup_{\substack{f:\Omega\to\R\\
\var f\neq0
}
}\frac{\var f}{\Dirichlet_{\text{r}}f}.\label{eq:t_rel}
\end{equation}

The following theorem provides an upper bound on the relaxation time:
\begin{thm}
\label{thm:gap_reservoir}Consider a noncooperative kinetically constrained
lattice gas on a finite box $\Lambda=[L]^{d}$ with reservoirs (see
equation (\ref{eq:generator_with_reservoir})) and empty boundary conditions. Fix a mobile cluster
$\MC$ of size $\MCsize$. Then for any $f:\Omega\to\R$, 
\[
\var f\le Cq^{-\MCsize-1}L^{2}\,\Dirichlet_{\text{r}}f,
\]
where the variance is taken with respect to the equilibrium $\mu$
and $\Dirichlet_{\text{r}}$ is the associated Dirichlet form given
in equation (\ref{eq:dirichlet_with_reservoir}).
\end{thm}

\subsection{Proof}

We will first prove Theorem \ref{thm:gap_reservoir} when $q\le\frac{1}{2}$,
and then briefly explain how to adapt the proof for $q>\frac{1}{2}$.

We follow the steps of \cite{BT04}---for any $x\in\Lambda$, we
will define a multistep move that creates a mobile cluster at the
boundary and uses it in order to flip the occupation at $x$. We will
then prove the theorem using this multistep move together with the
inequality
\begin{equation}
\var f\le q(1-q)\mu\left[\sum_{z\in\Lambda}(\grad_{z}f)^{2}\right].\label{eq:unconstrained_poincare_glauber}
\end{equation}

\begin{lem}
\label{lem:flip_reservoir}For any $z\in\Lambda$, there exists a
$T$-step move $\text{Flip}_{z}=\left((\eta_{t}),(x_{t}),(e_{t})\right)$
such that:
\begin{enumerate}
\item $\Dom\text{Flip}_{z}=\Omega_{\Lambda}$.
\item For any $\eta$, the final configuration is given by $\eta_{T}(\eta)=\eta^{z}$.
\item $T\le CL$.
\item The information loss $\Loss\text{Flip}_{z}\le C$.
\item The energy barrier $\EBarrier\text{Flip}_{z}\le\MCsize+1$.
\item For any $t\in\{0,\dots,T\}$, $x(t)\in z+\Delta$, where $\Delta\subset\Z^{d}$
is fixed and $\left|\Delta\right|\le CL$.
\item Each site $x\in\Lambda$ is changed a bounded number of times, i.e.,
$\left\{ t:x_{t}=x\right\} \le C$.
\end{enumerate}
\end{lem}

\begin{proof}
Let $\overline{z}=z-e_{1}\cdot z$, and consider the configuration
$\overline{\eta}$ defined on the infinite lattice as follows
\begin{equation}
\overline{\eta}(y)=\begin{cases}
\eta(y) & \text{if }y\in\Lambda,\\
1-\eta(z) & \text{if }y=\overline{z},\\
0 & \text{if }y\in\overline{z}-le_{1}+\MC,\\
1 & \text{otherwise}.
\end{cases}\label{eq:reservoir_extension}
\end{equation}

We will define a $\overline{T}$-step move $\overline{M}$ operating
on this configuration by composing exchange and translation moves
as follows---
\begin{enumerate}
\item Using the mobile cluster $\overline{z}-le_{1}+\MC$, apply the exchange
move $\text{Ex}_{1}(\overline{z}-le_{1}+\MC)$ (Definition (\ref{def:exchange_move}))
in order to exchange $\overline{z}$ with $\overline{z}+e_{1}$.
\item Apply the translation move $\text{Tr}_{1}(\overline{z}-le_{1}+\MC)$
(Definition (\ref{def:translation_move})) in order to move the cluster
$\overline{z}-le_{1}+\MC$ one step to the right.
\item Continue to apply these two moves alternatingly until reaching $x$,
i.e., 
\[
\text{Tr}_{1}(y_{k}+\MC)\circ\text{Ex}_{1}(y_{k}+\MC)\circ\dots\circ\text{Tr}_{1}(y_{1}+\MC)\circ\text{Ex}_{1}(y_{1}+\MC)\circ\text{Tr}_{1}(y_{0}+\MC)\circ\text{Ex}_{1}(y_{0}+\MC),
\]
where $y_{i}=\overline{z}-le_{1}+ie_{1}$ for all $i$, and $k$ is
chosen such that $y_{k}=z-2e_{1}$.
\item Apply the exchange move $\text{Ex}_{1}(y_{k}+e_{1}+\MC)$ in order
to exchange $y_{k}+e_{1}$ with $z$.
\item Wind back the exchanges and translations of step 3 and move the mobile
cluster back to $\overline{z}-le_{1}+\MC$.
\end{enumerate}
Putting everything together, we obtain 
\begin{gather*}
\overline{M}=\text{Ex}_{1}(y_{0}+\MC)\circ\text{Tr}_{-1}(y_{1}+\MC)\circ\dots\circ\text{Ex}_{1}(y_{k}+\MC)\circ\text{Tr}_{-1}(y_{k+1}+\MC)\circ\text{Ex}_{1}(y_{k+1}+\MC)\\
\circ\text{Tr}_{1}(y_{k}+\MC)\circ\text{Ex}_{1}(y_{k}+\MC)\circ\dots\circ\text{Tr}_{1}(y_{0}+\MC)\circ\text{Ex}_{1}(y_{0}+\MC).
\end{gather*}

We have thus constructed a multistep move $\overline{M}$ with the
following properties:
\begin{enumerate}
\item $\overline{\eta}\in\Dom\overline{M}$ for any $\eta\in\Omega_{\Lambda}$.
\item $\overline{M}$ is compatible with the transposition exchanging $\overline{z}$
and $z$.
\item $\overline{T}\le CL$.
\item $\Loss\overline{M}=0$ and $\EBarrier\overline{M}=0$.
\item All exchanges occur in a tube $\overline{z}+[-l,L]\times[-l,l]^{d-1}$
for some (large enough) fixed $l$.
\end{enumerate}
The move $\text{Flip}_{z}$ that we construct will simply be the restriction
of $\overline{M}$ to $\Lambda$---if we denote $\overline{M}=(\overline{\eta}_{t},\overline{x}_{t},\overline{e}_{t})$,
then $\text{Flip}_{z}$ will be such that, for any $y\in\Lambda$,
\[
\eta_{t}(y)=\overline{\eta}_{t}(y).
\]

All that is left is to verify that this move satisfies the required
properties:
\begin{enumerate}
\item It is well-defined on the entire $\Omega_{\Lambda}$---for any $\eta\in\Omega_{\Lambda}$
we know that $\overline{\eta}$ defined above is in $\Dom\overline{M}$.
In addition, a transition in $\overline{M}$ outside $\Lambda$ does
not change $\eta_{t}$, a transition on the boundary corresponds to
a reservoir term for $\eta_{t}$, and a transition inside $\Lambda$
which is allowed for $\overline{\eta}_{t}$ is certainly allowed for
$\eta_{t}$. This means that all transitions in $\text{Flip}_{x}$
are allowed, making it a valid move.
\item Since $\overline{z}\notin\Lambda$ and $\overline{\eta}(\overline{z})=1-\eta(z)$,
the fact that $\overline{M}$ is compatible with the transposition
exchanging $\overline{z}$ and $z$ implies that the final configuration
of $\text{Flip}_{z}$ is $\eta^{z}$.
\item $T=\overline{T}\le CL$.
\item In order to reconstruct $\overline{\eta}_{t}$ from $\eta_{t}$ it
is enough to know the occupation at some finite box to the left of
$\overline{z}$. Since $\overline{M}$ has $0$ loss of information,
the size of this box bounds the loss of information.
\item The number of vacancies in $\eta_{t}$ is certainly smaller than that
of $\overline{\eta}_{t}$, which exceeds the number of vacancies of
$\eta$ by at most $\MCsize+1$.
\item Choosing $\Delta=\overline{z}+[-l-L,L]\times[-l,l]^{d-1}$ will suffice.
\item Since the exchange and translation moves operate locally, a site $z$
could be ``touched'' by a bounded number of such moves, each of
which being able to change $z$ a bounded number of times. \qedhere
\end{enumerate}
\end{proof}
We will now use Lemma \ref{lem:flip_reservoir} in order to prove
Theorem \ref{thm:gap_reservoir}.
Start by considering, for each $z\in\Lambda$, the $T$-step move
$\text{Flip}_{z}=(\eta^{z},x^{z},e^{z})$, and using it in order to
write
\begin{align*}
(\grad_{z}f)^{2} & =\left(\sum_{t=0}^{T-1}\grad_{t}f(\eta_{t}^{z})\right)^{2}\\
 & \le CL\,\sum_{t=0}^{T-1}\left(\grad_{t}f(\eta_{t}^{z})\right)^{2},
\end{align*}
where $\grad_{t}$ stands for $\grad_{x_{t}^{z},x_{t}^{z}+e_{t}^{z}}$
for a bulk exchange ($\eta_{t+1}=\eta_{t}^{x_{t}^{z},x_{t}^{z}+e_{t}^{z}}$),
or $\grad_{x_{t}^{z}}$ for a boundary flip ($\eta_{t+1}=\eta_{t}^{x_{t}^{z}}$).

Then by equation \ref{eq:unconstrained_poincare_glauber} 
\begin{multline*}
\var f\le CLq(1-q)\mu\left[\sum_{z\in\Lambda}\sum_{t=0}^{T-1}\left(\grad_{t}f(\eta_{t}^{z})\right)^{2}\right]\\
=CLq\sum_{\eta\in\Omega_{\Lambda}}\mu(\eta)\sum_{z\in\Lambda}\sum_{t}\sum_{\eta'\in\Omega_{\Lambda}}\sum_{x\in z+\Delta}\sum_{e}\One_{\text{bulk exchange}}\One_{x_{t}^{z}(\eta)=x}\One_{e_{t}^{z}(\eta)=e}\One_{\eta_{t}(\eta)=\eta'}\,c_{x,x+e}(\eta')\left(\grad_{x,x+e}f(\eta')\right)^{2}\\
+CLq\sum_{\eta\in\Omega_{\Lambda}}\mu(\eta)\sum_{z\in\Lambda}\sum_{t}\sum_{\eta'\in\Omega_{\Lambda}}\sum_{x\in\partial\Lambda\cap(z+\Delta)}\One_{\text{bounday flip}}\One_{x_{t}^{z}(\eta)=x}\One_{\eta_{t}=\eta'}\,\left(\grad_{x}f(\eta')\right)^{2}\\
\le CLq\sum_{x\in\Lambda}\sum_{e}\sum_{\eta'\in\Omega_{\Lambda}}\mu(\eta')c_{x,x+e}(\eta')\left(\grad_{x,x+e}f(\eta')\right)^{2}\,\sum_{\eta\in\Omega_{\Lambda}}\frac{\mu(\eta)}{\mu(\eta')}\sum_{z\in x-\Delta}\sum_{t}\One_{x_{t}^{z}(\eta)=x}\One_{\eta_{t}(\eta)=\eta'}\\
+CL\sum_{x\in\partial\Lambda}\sum_{\eta'\in\Omega_{\Lambda}}\mu(\eta')c_{x}(\eta')\left(\grad_{x}f(\eta')\right)^{2}\sum_{\eta\in\Omega_{\Lambda}}\frac{\mu(\eta)}{\mu(\eta')}\sum_{z\in x-\Delta}\sum_{t}\One_{x_{t}^{z}(\eta)=x}\One_{\eta_{t}=\eta'}.
\end{multline*}

We will now use the properties of $\text{Flip}_{z}$ in order to bound
the different terms above. First, since we assume $q\le\frac{1}{2}$,
\[
\frac{\mu(\eta)}{\mu(\eta')}\le q^{-\EBarrier(\text{Flip}_{z})}=q^{-\MCsize-1}.
\]

The bound on the loss of information allows us to write $\sum_{\eta\in\Omega_{\Lambda}}\One_{\eta_{t}(\eta)=\eta'}\le C$.

The last property of the flip move implies that $\sum_{t=0}^{T}\One_{x_{t}^{z}(\eta)=x}\le C$.

Putting everything together, we obtain 
\begin{eqnarray*}
\var f & \le & CLq^{-\MCsize-1}\left|\Delta\right|\sum_{x\in\Lambda}\sum_{e}\sum_{\eta'\in\Omega_{\Lambda}}\mu(\eta')c_{x,x+e}(\eta')\left(\grad_{x,x+e}f(\eta')\right)^{2}\\
 &  & +CLq^{-\MCsize-1}\left|\Delta\right|\sum_{x\in\partial\Lambda}\sum_{\eta'\in\Omega_{\Lambda}}\mu(\eta')c_{x}(\eta')\left(\grad_{x}f(\eta')\right)^{2}\\
 & \le & CL^{2}q^{-N-1}\Dirichlet_{\Lambda}f.
\end{eqnarray*}
This concludes the proof when $q\le\frac{1}{2}$.

The case $q>\frac{1}{2}$ could be thought of as a negative temperature
setting, so the relevant quantity is the \emph{negative} energy barrier---rather
than counting the excess vacancies, we should count the excess particles.
By changing the definition of $\overline{\eta}$ given in equation
(\ref{eq:reservoir_extension}) such that $\overline{\eta}(y)=0$
if $y\notin\Lambda\cup\{\overline{z}\}$, we can construct the $\text{Flip}_{z}$
in the same manner, such that at each $t$ the number of particles
in $\eta_{t}$ exceeds the number of particles in $\eta$ by at most
$1$. The only estimate that changes is that of $\frac{\mu(\eta)}{\mu(\eta')}$,
which becomes $\frac{\mu(\eta)}{\mu(\eta')}\le(1-q)^{-1}$, and the
rest of the proof follows.\qed 

\section{\label{sec:relaxation_closed}Relaxation time in a closed system}

In this section we consider models on a finite box $\Lambda=[L]^{d}$,
with no reservoirs. In this setting the total number of particles
is fixed, hence $\mu$ cannot be ergodic. Moreover, even if we condition
$\mu$ to some fixed number of vacancies $k$, the measure that we
obtain is in general not ergodic due to the constraint.

In particular, at least if $q$ is not too large,
one may construct \emph{blocked configuration}. These are configurations
where no particle is allowed to jump, and therefore do not change
during the dynamics (see, e.g., Figure \ref{fig:blocked_configuration}).
If $k<\left(\frac{L}{R+1}\right)^{d}$ (where $R$ is the range of
the constraint), we may place the vacancies such that no two empty sites
are at distance less than $R$. Since the model is nondegenerate the
constraint is not satisfied for the edges adjacent to a vacancy, and
the configuration is indeed blocked.

For noncooperative models, we note that two configurations containing
a mobile cluster, at least for $k$ large enough, are always in the
same ergodic component---consider two configurations $\eta$ and
$\eta'$ with $k$ vacancies, each containing a mobile cluster, $x+\MC$
and $x'+\MC'$ respectively. Assuming $k>\left|\MC\right|+\left|\MC'\right|$,
we may use the translation and exchange moves on $\eta$ with the
cluster $x+\MC$ in order to move vacancies to $x'+\MC'$. Then we
use the translation and exchange moves with the cluster $x'+\MC'$
to move around all other vacancies to their locations in $\eta'$.

We therefore define the ergodic configurations as follows:
\begin{defn}
\label{def:ergodic_component}Consider a family of mobile clusters
$\left\{ \MC_{1},\dots,\MC_{m}\right\} $. The set of ergodic configurations
with $k$ vacancies, denoted $\Omega_{k}$, is given by all configurations
$\eta$ containing $k$ vacancies connected to a configuration that
contains an empty translation of a mobile cluster. More precisely,
$\eta\in\Omega_{k}$ if it contains $k$ vacancies, and there exists
a $T$-step move $M=((\eta_{t}),(x_{t}),(e_{t}))$, a site $x\in\Lambda$,
and some $i\in[m]$, such that $\eta\in\Dom M$ and all sites of $x+\MC_{i}$
are empty for the configuration $\eta_{T}(\eta)$.

The \emph{equilibrium measure} $\mu_{k}$ is the uniform measure on $\Omega_{k}$.
We denote in this section $\mu=\mu_k$.
\end{defn}

The discussion above implies the following fact:
\begin{fact}
For any family of mobile clusters $\left\{ \MC_{1},\dots,\MC_{m}\right\} $,
and any $k>2\max_{i=1}^{m}\left|\MC_{i}\right|$, the measure $\mu$
is ergodic.
\end{fact}

\begin{example}
Consider the model of Example \ref{exa:1d}, and the family of mobile
clusters $\{ \{1,2\}, \linebreak \{1,3\}\}$. If a configuration $\eta$
does not contain an empty translation of either cluster, it is blocked,
since all allowed transitions for the dynamics involve two vacancies
at distance at most $2$. Therefore, the ergodic configurations in
this models are those containing an empty translation of $\{1,3\}$
or $\{1,2\}$.
\end{example}

\begin{example}
\label{exa:ergodic_component_2d}In the model introduced in Example
\ref{exa:2d} the ergodic component is more complicated. One can find
configurations that are not blocked but still not ergodic, or configurations
which are ergodic but do not contain a mobile cluster of size smaller
than $L$. An explicit description of $\Omega_{k}$ for this model
seems to be much more difficult to find than the $1$ dimensional case.
See Figure \ref{fig:ergodic_component_2d}.

\begin{figure}
\begin{tabular}{cc}
\begin{tikzpicture}[scale=0.4, every node/.style={scale=0.6}]
	\draw[step=1,gray] (0,0) grid +(8,8);
	\foreach \x in {0,...,7}{
		\foreach \y in {0,...,7}{
			\fill[black,shift={(0.5,0.5)}] (\x,\y) circle (0.4);
		}
	}
	
	\fill[white,shift={(0.5,0.5)}] (1,0) circle (0.45);
	\fill[white,shift={(0.5,0.5)}] (6,3) circle (0.45);
	\fill[white,shift={(0.5,0.5)}] (3,6) circle (0.45);
	\fill[white,shift={(0.5,0.5)}] (1,5) circle (0.45);
\end{tikzpicture} & \begin{tikzpicture}[scale=0.4, every node/.style={scale=0.6}]
	\draw[step=1,gray] (0,0) grid +(8,8);
	\foreach \x in {0,...,7}{
		\foreach \y in {0,...,7}{
			\fill[black,shift={(0.5,0.5)}] (\x,\y) circle (0.4);
		}
	}
	
	\fill[white,shift={(0.5,0.5)}] (6,3) circle (0.45);
	\fill[white,shift={(0.5,0.5)}] (3,5) circle (0.45);
	\fill[white,shift={(0.5,0.5)}] (1,5) circle (0.45);
	\fill[white,shift={(0.5,0.5)}] (6,2) circle (0.45);
\end{tikzpicture}\tabularnewline
{\small{}(a)\medskip{}
} & {\small{}(b)}\tabularnewline
\begin{tikzpicture}[scale=0.4, every node/.style={scale=0.6}]
	\draw[step=1,gray] (0,0) grid +(8,8);
	\foreach \x in {0,...,7}{
		\foreach \y in {0,...,7}{
			\fill[black,shift={(0.5,0.5)}] (\x,\y) circle (0.4);
		}
	}
	
	\fill[white,shift={(0.5,0.5)}] (1,1) circle (0.45);
	\fill[white,shift={(0.5,0.5)}] (1,2) circle (0.45);
	\fill[white,shift={(0.5,0.5)}] (2,1) circle (0.45);
	\fill[white,shift={(0.5,0.5)}] (2,2) circle (0.45);
\end{tikzpicture} & \begin{tikzpicture}[scale=0.4, every node/.style={scale=0.6}]
	\draw[step=1,gray] (0,0) grid +(8,8);
	\foreach \x in {0,...,7}{
		\foreach \y in {0,...,7}{
			\fill[black,shift={(0.5,0.5)}] (\x,\y) circle (0.4);
		}
	}
	
	\fill[white,shift={(0.5,0.5)}] (1,5) circle (0.45);
	\fill[white,shift={(0.5,0.5)}] (1,6) circle (0.45);
	\fill[white,shift={(0.5,0.5)}] (2,1) circle (0.45);
	\fill[white,shift={(0.5,0.5)}] (5,2) circle (0.45);
	\fill[white,shift={(0.5,0.5)}] (7,2) circle (0.45);
\end{tikzpicture}\tabularnewline
{\small{}(c)} & {\small{}(d)}\tabularnewline
\end{tabular}

\caption{\label{fig:ergodic_component_2d}A few configurations in the model
of Example \ref{exa:2d} defined on a finite box. The mobile cluster of this model is a $2\times2$
square (see Example \ref{exa:tr_2d} and Figure \ref{fig:mc_2d}).
Configuration (a) is blocked hence not ergodic, configuration (b) is not blocked but
still not ergodic, configuration (c) contains a mobile cluster hence
ergodic, and configuration (d) is ergodic even though no small region
contains a mobile cluster. See Example \ref{exa:ergodic_component_2d}.}

\end{figure}
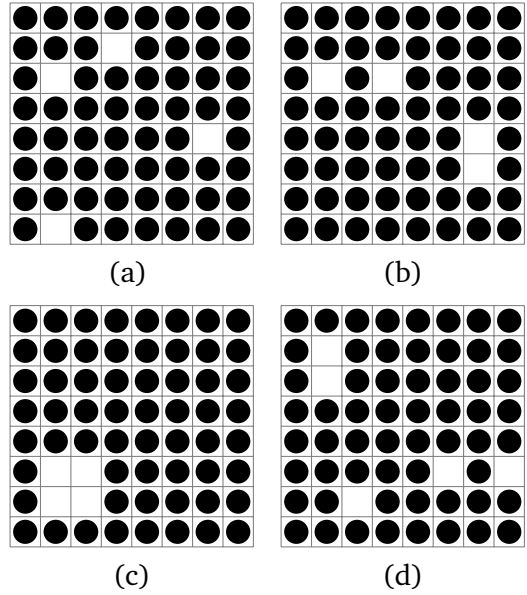
\end{example}

In view of these examples, we will restrict our discussion to models
with easily identifiable set of ergodic configurations:
\begin{hypothesis}
\label{hyp:MC_in_ergodic_component}There exists a finite family of
mobile clusters, $\left\{ \MC_{1},\dots,\MC_{m}\right\} $, such that
\[
\Omega_{k}=\left\{ \eta:\text{there exist }x\in\Lambda\text{ and }i\in[m]\text{ for which }x+\MC_{i}\text{ is in }\Lambda\text{ and empty}\right\} .
\]
\end{hypothesis}

Fix $k>\max_{i=1}^{m}\left|\MC_{i}\right|$, so $\Omega_{k}$ is nonempty
and the measure $\mu$ is well defined. The Dirichlet form associated
with the generator (\ref{eq:generator}) and the (reversible)
measure $\mu$ is given by
\begin{equation}
\Dirichlet f=\mu\left[\sum_{\substack{x,y\in\Lambda\\
x\sim y
}
}c_{x,y}(\grad_{x,y}f)^{2}\right].\label{eq:dirichlet_closed_system}
\end{equation}

The result of this section is a bound on the relaxation time of $1$
dimensional models satisfying Hypothesis \ref{hyp:MC_in_ergodic_component}:
\begin{thm}
\label{thm:gap_closed}Consider a noncooperative kinetically constrained
lattice gas with occupied boundary conditions in one dimension satisfying 
Hypothesis \ref{hyp:MC_in_ergodic_component},
and let $k=\floor{qL}$ for some $q\in(0,1)$. Then for $L$ large
enough and any $f:\Omega_{k}\to\R$ 
\[
\var f\le Cq^{C}L^{2}\,\Dirichlet f,
\]
where the variance is taken with respect to $\mu=\mu_{k}$ and $\Dirichlet$
is the associated Dirichlet form given in equation (\ref{eq:dirichlet_closed_system}).
\end{thm}

\subsection{Proof}

The overall scheme of the proof is similar to that of \cite{GoncalvesLandimToninelli}---we
first create many mobile clusters, and then use them in order to exchange
the occupation of pairs of sites. This will allow us to compare our
model with the simple exclusion process on the complete graph. The
main difference between the proof here and the one presented in \cite{GoncalvesLandimToninelli}
is that the creation of the mobile clusters is accomplished without
resorting to a perturbed model.

We start with a few definitions, which will depend on a fixed arbitrary
mobile cluster $\MC$ of size $\MCsize$, and an integer $\boxsize>\frac{2N}{q}$
such that $\MC_{i}\subset[\boxsize]$ for all $i\in\{1,\dots,k\}$.
\begin{defn}
A \emph{box} (of size $\boxsize$) is a subset of $\Lambda$ of the
type $\boxsize i+[\boxsize]$, for $i\in\Z$. We may assume that $\frac{L}{\boxsize}\in\N$
by the same monotonicity argument as in \cite[Remark 3.1]{MST19KA},
and denote the set of boxes 
\[
\cB=\left\{ \boxsize i+[\boxsize],\,i\in\Z\cap[0,L/\boxsize-1]\right\} .
\]
\end{defn}

\begin{defn}
A \emph{good box} is a box containing an empty translation of $\MC$.

A \emph{pregood box} is a box containing at least $\MCsize$ vacancies
(recall $\MCsize=\left|\MC\right|$).

We denote by $G$ the event that at least $k_{0}=\floor{\boxsize^{-N}\left(\frac{k}{4\boxsize}-1\right)}$
boxes are good. We assume $L$ (and therefore $k$) large enough so
that $k_{0}>0$.
\end{defn}

\begin{claim}
\label{claim:many_pregood_boxes}For any $\eta\in\Omega_{k}$, at
least $\frac{k}{2\boxsize}$ boxes are pregood.
\end{claim}

\begin{proof}
Let $n_{v}$ be the number of boxes containing exactly $v$ vacancies,
so the number of pregood boxes is $\sum_{v=N}^{\boxsize}n_{v}$. Then
\begin{align*}
k & =\sum_{v=0}^{\boxsize}vn_{v}=\sum_{v=0}^{\MCsize-1}vn_{v}+\sum_{v=\MCsize}^{\boxsize}vn_{v}\\
 & \le N\left|\cB\right|+\boxsize\sum_{v=\MCsize}^{\boxsize}n_{v}\le\frac{k}{2}+\boxsize\cdot\#\text{pregood boxes}.\qedhere
\end{align*}
\end{proof}
\begin{defn}
\label{def:tagged_permutation}Let $\Sigma$ be the set whose elements
are of the type $s=(o,\sigma)$, for $o\in\{+,-\}$ and $\sigma=(\sigma_{B})_{B\in\cB},$
where $\sigma_{B}$ is a permutation of the sites of $B$ for any
box $B\in\cB$. 

For a configuration $\eta\in\Omega_{k}$ and $s\in\Sigma$, we construct
the configuration $s\eta$ as follows:
\begin{enumerate}
\item Find the the first mobile cluster in the orientation $o$, that is,
the site $z\in\Lambda$ together with $i\in\{1,\dots,k\}$ such that:
\begin{enumerate}
\item $z+\MC_{i}$ is empty for some $i\in\{1,\dots,k\}$.
\item $z$ is the leftmost site satisfying (a) if $o=+$, and the rightmost
if $o=-$. Differently stated, for any $y\neq z$ such that $y+\MC_{j}$
is empty for some $j\in\{1,\dots,k\}$, $oz<oy$.
\end{enumerate}
\item Identify the set $\cB_{o}$ of boxes after $z$, that is, the boxes
$B\in\cB$ in which all sites are strictly to the right of $z+\MC_{i}$
if $o=+$, or strictly to its left in the case $o=-$.
\item For $x\in\Lambda$, denoting by $B$ the box containing $x$,
\[
s\eta(x)=\begin{cases}
\eta(x) & \text{if }B\notin\cB_{o},\\
\eta(\sigma_{B}^{-1}x) & \text{if }B\in\cB_{o}.
\end{cases}
\]
\end{enumerate}
\end{defn}

\begin{observation}
The action defined above is bijective---for any $s\in\Sigma$ we
can define $s^{-1}\in\Sigma$ by inverting each permutation and keeping
the orientation fixed. Then $ss^{-1}\eta=\eta$ for any $\eta\in\Omega_{k}$.
\end{observation}

\begin{claim}
Fix $\eta\in\Omega_{k}$. Then 
\[
\frac{\left|\left\{ s\in\Sigma:s\eta\in G\right\} \right|}{\left|\Sigma\right|}\ge\frac{1}{4}.
\]
\end{claim}

\begin{proof}
We use the notation of Definition \ref{def:tagged_permutation}. $\cup_{o\in\{\pm\}}\cB_{o}$
contains all boxes, except for a maximum of $2$ boxes containing
sites of the mobile cluster. By Claim \ref{claim:many_pregood_boxes},
at least $\frac{k}{2\boxsize}-2$ of them are pregood. Hence, there
is an orientation $o^{\star}\in\{+,-\}$, such that the number of
pregood boxes in $\cB_{o^{\star}}$ is at least $\frac{\nicefrac{k}{2\boxsize}-2}{2}$.

Let $s=(o,\sigma)$ be an element of $\Sigma$ chosen uniformly at
random. Equivalently, we can say that $o$ is chosen uniformly at
random from $\{+,-\}$ and each permutation in $\sigma$ is chosen
uniformly at random, all independently of one another. As we have
seen above, under this measure, denoting by $p$ the number of boxes
in $\cB_{o}$ that are pregood for $\eta$,
\[
\P\left[p\ge\frac{k}{4\boxsize}-1\right]\ge\P[o=o^{\star}]=\frac{1}{2}.
\]

For each box $B\in\cB_{o}$ which is pregood for $\eta$, the probability
that $B$ is good for $s\eta$ is at least $\boxsize^{-N}$. Hence,
conditioning on $p\ge\frac{k}{4l}-1$, the number of good boxes for
$s\eta$ is dominating a binomial random variable of parameters $\frac{k}{4l}-1$
and $\boxsize^{-N}$. The median of the latter is $\boxsize^{-N}\left(\frac{k}{4\boxsize}-1\right)=k_{0}$,
hence 
\[
\P\left[\#\text{good boxes for }s\eta\ge k_{0}|p\ge\frac{k}{4\boxsize}-1\right]\ge\frac{1}{2}.
\]
This concludes the proof.
\end{proof}
In order to bound the variance of $f$, we start by writing 
\begin{eqnarray*}
\var f & = & \frac{1}{2}\sum_{\eta,\eta'\in\Omega_{k}}\mu(\eta)\mu(\eta')\left(f(\eta)-f(\eta')\right)^{2}\\
 & = & \frac{1}{2}\sum_{\eta,\eta'\in\Omega_{k}}\mu(\eta)\mu(\eta')\,\frac{1}{\left|\left\{ s\in\Sigma:s\eta\in G\right\} \right|^2}\sum_{s\in\Sigma}\One_{s\eta\in G}\,\sum_{s'\in\Sigma}\One_{s'\eta'\in G}\,\left(f(\eta)-f(\eta')\right)^{2}\\
 & \le & \frac{C}{\left|\Sigma\right|^{2}}\sum_{s,s'\in\Sigma}\sum_{\eta,\eta'}\mu(\eta)\mu(\eta')\One_{s\eta\in G}\One_{s'\eta'\in G}\left(f(\eta)-f(\eta')\right)^{2}\\
 & = & \frac{C}{\left|\Sigma\right|^{2}}\sum_{s,s'\in\Sigma}\sum_{\eta,\eta'}\mu(\eta)\mu(\eta')\One_{s\eta\in G}\One_{s'\eta'\in G}\left(f(\eta)-f(s\eta)+f(s\eta)-f(s'\eta')+f(s'\eta')-f(\eta')\right)^{2}\\
 & \le & \frac{C}{\left|\Sigma\right|^{2}}\sum_{s,s'\in\Sigma}\sum_{\eta,\eta'}\mu(\eta)\mu(\eta')\One_{s\eta\in G}\One_{s'\eta'\in G}\left(f(\eta)-f(s\eta)\right)^{2}\\
 &  & +\frac{C}{\left|\Sigma\right|^{2}}\sum_{s,s'\in\Sigma}\sum_{\eta,\eta'}\mu(\eta)\mu(\eta')\One_{s\eta\in G}\One_{s'\eta'\in G}\left(f(s\eta)-f(s'\eta')\right)^{2}\\
 &  & +\frac{C}{\left|\Sigma\right|^{2}}\sum_{s,s'\in\Sigma}\sum_{\eta,\eta'}\mu(\eta)\mu(\eta')\One_{s\eta\in G}\One_{s'\eta'\in G}\left(f(s'\eta')-f(\eta')\right)^{2}\\
 & \le & \frac{C}{\left|\Sigma\right|}\sum_{s}\sum_{\eta}\mu(\eta)\left(f(\eta)-f(s\eta)\right)^{2}+\frac{C}{\left|\Sigma\right|^{2}}\sum_{s,s'\in\Sigma}\sum_{\eta,\eta'}\mu(\eta)\mu(\eta')\One_{s\eta\in G}\One_{s'\eta'\in G}\left(f(s\eta)-f(s'\eta')\right)^{2}\\
 & = & \text{I}+\text{II}.
\end{eqnarray*}
In order to finish the proof of the theorem, it is left to show that
\begin{align}
\text{I} & \le Cq^{-C}L^{2}\Dirichlet_{\Lambda}f,\label{eq:inequality_I}\\
\text{II} & \le Cq^{-C}L^{2}\Dirichlet_{\Lambda}f.\label{eq:inequality_II}
\end{align}

Let us start with inequality (\ref{eq:inequality_I}).
\begin{claim}
For any $s=(o,\sigma)\in\Sigma$ and $z\in\Lambda$ there exists a
$T$-step move $M_{s,z}=((\eta_{t}),(x_{t}),(e_{t}))$ satisfying:
\begin{enumerate}
\item $\Dom M_{s}=\left\{ \eta\in\Omega_{k}:z\text{ is the first mobile cluster in \ensuremath{\eta\text{ for the orientation }}}o\right\} $.
\item $\eta_{T}(\eta)=s\eta$ for any $\eta\in\Dom M_{s}$.
\item $T\le Cl^{3}L$.
\item $\Loss M_{s}=0$.
\item Each site $x\in\Lambda$ is exchanged at most $C\boxsize^{3}$ times.
Moreover, 
\[
\left|\left\{ t\text{ such that }x_{t}(\eta)=x\text{ for some }\eta\in\Dom M_{s,z}\right\} \right|\le C\boxsize^{3}.
\]
\end{enumerate}
\end{claim}

\begin{proof}
Assume for simplicity $o=+$, the case $o=-$ is analogous.

We start with the mobile cluster at $z$, and use the translation
move (Definition \ref{def:translation_move}) $L-\boxsize-z$ times
in order to move it to the box $[L-2\boxsize+1,L-\boxsize]$. The
permutation $\sigma_{[L-\boxsize+1,L]}$ can be decomposed as a product
of at most $C\boxsize^{2}$ nearest neighbor transpositions (see,
e.g., \cite[Section 5.2.2]{Knuth73}). We apply them one by one, where
at each step in order to exchange $L-\boxsize+x$ with $L-\boxsize+x+1$
we move the cluster $x$ times to the right using the translation
move (Definition \ref{def:translation_move}), then exchange $L-\boxsize+x$
with $L-\boxsize+x+1$ using the exchange move (Definition \ref{def:exchange_move}),
and finally move the cluster $x$ times to the left. Each transposition
takes $2xT_{\text{Tr}}+T_{\text{Ex}}<Cl$ steps.

Once the permutation $\sigma_{[L-\boxsize+1,L]}$ has been applied,
we move the cluster $\boxsize$ steps to the left, to the box $[L-3\boxsize+1,L-2\boxsize]$,
and apply as before the permutation $\sigma_{[L-2\boxsize+1,L-\boxsize]}$
to the box $[L-2\boxsize+1,L-\boxsize]$. Continue in the same manner
until all boxes in $\cB_{+}$ are rearranged, and move the cluster
back to $z$.

The verification of 2-5 is immediate.
\end{proof}
We now use the move $M_{s,z}=((\eta_{t}^{s,z}),(x_{t}^{s,z}),(e_{t}^{s,z}))$
in order to bound the term I: for any $s\in\Sigma$,
%\begin{align*}
%\sum_{\eta}\mu(\eta)\left(f(\eta)-f(s\eta)\right)^{2} & =\sum_{\eta}\mu(\eta)\sum_{z\in\Lambda}\One_{\eta\in\Dom M_{s,z}}\left(f(\eta)-f(s\eta)\right)^{2}\\
% & =\sum_{\eta}\mu(\eta)\sum_{z\in\Lambda}\One_{\eta\in\Dom M_{s,z}}\left(\sum_{t=0}^{T-1}\grad_{x_{t}^{s,z},x_{t}^{s,z}+e_{t}^{s,z}}f(\eta_{t}^{s,z})\right)^{2}\\
% & \le C\sum_{\eta}\mu(\eta)\sum_{z\in\Lambda}T\sum_{\eta'\in\Omega}\sum_{x\in\Lambda}\One_{\eta\in\Dom M_{s,z}}\sum_{t=0}^{T-1}\One_{\eta'=\eta_{t}^{s,z}}\One_{x=x_{t}^{s,z}}c_{x,x+1}(\eta')\left(\grad_{x,x+1}f(\eta')\right)^{2}\\
% & \le C\boxsize^{6}L^{2}\sum_{\eta'}\mu(\eta')\sum_{x\in\Lambda}c_{x,x+1}(\eta')\left(\grad_{x,x+1}f(\eta')\right)^{2}=C\boxsize^{6}L^{2}\Dirichlet f.
%\end{align*}
\begin{multline*}
\sum_{\eta}\mu(\eta)\left(f(\eta)-f(s\eta)\right)^{2}  =\sum_{\eta}\mu(\eta)\sum_{z\in\Lambda}\One_{\eta\in\Dom M_{s,z}}\left(f(\eta)-f(s\eta)\right)^{2}\\
  =\sum_{\eta}\mu(\eta)\sum_{z\in\Lambda}\One_{\eta\in\Dom M_{s,z}}\left(\sum_{t=0}^{T-1}\grad_{x_{t}^{s,z},x_{t}^{s,z}+e_{t}^{s,z}}f(\eta_{t}^{s,z})\right)^{2}\\
  \le C\sum_{\eta}\mu(\eta)\sum_{z\in\Lambda}T\sum_{\eta'\in\Omega}\sum_{x\in\Lambda}\One_{\eta\in\Dom M_{s,z}}\sum_{t=0}^{T-1}\One_{\eta'=\eta_{t}^{s,z}}\One_{x=x_{t}^{s,z}}c_{x,x+1}(\eta')\left(\grad_{x,x+1}f(\eta')\right)^{2}\\
  \le C\boxsize^{6}L^{2}\sum_{\eta'}\mu(\eta')\sum_{x\in\Lambda}c_{x,x+1}(\eta')\left(\grad_{x,x+1}f(\eta')\right)^{2}=C\boxsize^{6}L^{2}\Dirichlet f.
\end{multline*}
Therefore 
\[
\frac{C}{\left|\Sigma\right|}\sum_{s}\sum_{\eta}\mu(\eta)\left(f(\eta)-f(s\eta)\right)^{2}\le C\boxsize^{6}L^{2}\Dirichlet f.
\]
For $q$ small we may choose $\boxsize<\frac{2N+1}{q}$ and inequality
(\ref{eq:inequality_I}) is satisfied. For $q$ large the $q$ and $\lambda$ dependence could be put it the constant
$C$, proving inequality (\ref{eq:inequality_I}) for all $q$.

We move to inequality (\ref{eq:inequality_II}). Start by noting that,
thanks to the bijectivity of $s$ and $s'$, we can change variables
in the sum to obtain 
\begin{align*}
\text{II} & =\frac{C}{\left|\Sigma\right|^{2}}\sum_{s,s'\in\Sigma}\sum_{\eta,\eta'}\mu(\eta)\mu(\eta')\One_{\eta\in G}\One_{\eta'\in G}\left(f(\eta)-f(\eta')\right)^{2}\\
 & =C\sum_{\eta,\eta'}\mu(\eta)\mu(\eta')\One_{\eta\in G}\One_{\eta'\in G}\left(f(\eta)-f(\eta')\right)^{2}.
\end{align*}
Since under the good event there are at least $k_{0}$ sites $x$
for which $x+\MC$ is empty, 
\begin{align*}
\text{II} & \le C\sum_{\eta,\eta'}\mu(\eta)\mu(\eta')\One_{\eta\in G}\One_{\eta'\in G}\,\frac{1}{k_{0}}\sum_{z\in\Lambda}\One_{z+\MC\text{ is empty for }\eta}\frac{1}{k_{0}}\sum_{z'\in\Lambda}\One_{z'+\MC\text{ is empty for }\eta'}\left(f(\eta)-f(\eta')\right)^{2}\\
 & \le\frac{C}{k_{0}^{2}}\sum_{\eta,\eta'}\mu(\eta)\mu(\eta')\sum_{z,z'}\One_{z+\MC\text{ is empty for }\eta}\One_{z'+\MC\text{ is empty for }\eta'}\left(f(\eta)-f(\eta')\right)^{2}.
\end{align*}

For $\eta$ such that $z+\MC$ is empty, let $\Theta_{z}\eta$ be
the outcome of $z$ translations moves to the left. That is, $\Theta_{z}$
is the permutation compatible with $\text{Tr}_{-1}(1+\MC)\circ\dots\circ\text{Tr}_{-1}(z+\MC)$.
We can then write $\text{II}$ as 
\begin{eqnarray*}
\text{II} & \le & \frac{C}{k_{0}^{2}}\sum_{\eta,\eta'}\mu(\eta)\mu(\eta')\sum_{z,z'}\One_{\eta(z+\MC)=0}\One_{\eta'(z'+\MC)=0}\left(f(\eta)-f(\Theta_{z}\eta)+f(\Theta_{z}\eta)-f(\Theta_{z'}\eta')+f(\Theta_{z'}\eta')-f(\eta')\right)^{2}\\
 & \le & \frac{C}{k_{0}^{2}}\sum_{\eta,\eta'}\mu(\eta)\mu(\eta')\sum_{z,z'}\One_{\eta(z+\MC)=0}\One_{\eta'(z'+\MC)=0}\left(f(\eta)-f(\Theta_{z}\eta)\right)^{2}\\
 &  & \qquad +\frac{C}{k_{0}^{2}}\sum_{\eta,\eta'}\mu(\eta)\mu(\eta')\sum_{z,z'}\One_{\eta(z+\MC)=0}\One_{\eta'(z'+\MC)=0}\left(f(\Theta_{z}\eta)-f(\Theta_{z'}\eta')\right)^{2}\\
 &  & \qquad +\frac{C}{k_{0}^{2}}\sum_{\eta,\eta'}\mu(\eta)\mu(\eta')\sum_{z,z'}\One_{\eta(z+\MC)=0}\One_{\eta'(z'+\MC)=0}\left(f(\Theta_{z'}\eta')-f(\eta')\right)^{2}.\\
 & \le & \frac{CL}{k_{0}^{2}}\sum_{\eta}\mu(\eta)\sum_{z}\One_{\eta(z+\MC)=0}\left(f(\eta)-f(\Theta_{z}\eta)\right)^{2}\\
 &  & \qquad +\frac{CL^{2}}{k_{0}^{2}}\sum_{\eta,\eta'}\mu(\eta)\mu(\eta')\One_{\eta(\MC)=0}\One_{\eta'(\MC)=0}\left(f(\eta)-f(\eta')\right)^{2}\\
 & = & \text{III}+\text{IV}.
\end{eqnarray*}
The term $\text{III}$ could be bounded using the $T$-step move \textbf{$M=((\eta_{t}),(x_{t}),(e_{t}))$}
resulted from the composition of $z$ translations to the left---it
is not difficult to see that $T\le CL$, that it has $0$ loss, and
that each edge is flipped a bounded number of times. Therefore 
\begin{align*}
\text{III} & \le\frac{CL^{2}}{k_{0}^{2}}\sum_{\eta}\mu(\eta)\sum_{z}\One_{\eta(z+\MC)=0}\sum_{\eta'}\sum_{x}\sum_{t=0}^{T-1}\One_{\eta'=\eta_{t}}\One_{x_{t}=x}c_{x,x+1}(\eta')\left(\grad_{x,x+1}f(\eta')\right)^{2}\\
 & \le\frac{CL^{3}}{k_{0}^{2}}\sum_{\eta'}\mu(\eta')c_{x,x+1}(\eta')\sum_{x}\left(\grad_{x,x+1}f(\eta')\right)^{2}\le\frac{CL^{2}}{k_{0}^{2}}\,L\Dirichlet f\le Cq^{-C}L\,\Dirichlet f.
\end{align*}

In order to estimate the last term $\text{IV}$, we need two ingredients---first,
let $\overline{\Omega}_{k-N}$ be the space of configurations on $\Lambda\setminus\MC$
with $k-N$ particles, endowed with the uniform measure $\overline{\mu}$.
Note that to any configuration $\eta\in\Omega_{k}$ in which $\MC$
is empty we can associate a configuration $\overline{\eta}\in\overline{\Omega}_{k-N}$
and vice versa. We may also define the function $\overline{f}:\overline{\Omega}\to\R$,
given by $\overline{f}(\overline{\eta})=f(\eta)$. Then
%\begin{align*}
%\text{IV} & \le\frac{CL^{2}}{k_{0}^{2}}\frac{\left|\overline{\Omega}_{k-N}\right|^{2}}{\left|\Omega_{k}\right|^{2}}\sum_{\overline{\eta},\overline{\eta}'}\overline{\mu}(\overline{\eta})\overline{\mu}(\overline{\eta}')\left(\overline{f}(\overline{\eta})-\overline{f}(\overline{\eta}')\right)^{2}\\
% & \le\frac{CL^{2}}{k_{0}^{2}}\frac{\left|\overline{\Omega}_{k-N}\right|^{2}}{\left|\Omega_{k}\right|^{2}}\sum_{\overline{\eta},\overline{\eta}'}\overline{\mu}(\overline{\eta})\overline{\mu}(\overline{\eta}')\left(\overline{f}(\overline{\eta})-\overline{f}(\overline{\eta}')\right)^{2}.
%\end{align*}
\[
\text{IV}=\frac{CL^{2}}{k_{0}^{2}}\frac{\left|\overline{\Omega}_{k-N}\right|^{2}}{\left|\Omega_{k}\right|^{2}}\sum_{\overline{\eta},\overline{\eta}'}\overline{\mu}(\overline{\eta})\overline{\mu}(\overline{\eta}')\left(\overline{f}(\overline{\eta})-\overline{f}(\overline{\eta}')\right)^{2}.
\]
 Note that the variance of $\overline{f}$ with respect to the measure
$\overline{\mu}$ is given by 
\[
\var_{\overline{\mu}}\overline{f}=\frac{1}{2}\sum_{\overline{\eta},\overline{\eta}'}\overline{\mu}(\overline{\eta})\overline{\mu}(\overline{\eta}')\left(\overline{f}(\overline{\eta})-\overline{f}(\overline{\eta}')\right)^{2}.
\]
We can therefore bound $\text{IV}$ using the relaxation time of the
simple exclusion process on the complete graph \cite{DiaconisSaloffCoste1993,DiaconisShahshahni1981},
expressed in the following Poincar\'e inequality:
\[
\var_{\overline{\mu}}\overline{f}\le\frac{1}{L-N}\sum_{\overline{\eta}}\overline{\mu}(\overline{\eta})\sum_{y,z\in\Lambda\setminus\MC}\left(\grad_{x,y}\overline{f}(\overline{\eta})\right)^{2}.
\]

Thus 
\begin{align*}
\text{IV} & \le\frac{CL}{k_{0}^{2}}\frac{\left|\overline{\Omega}_{k-N}\right|^{2}}{\left|\Omega_{k}\right|^{2}}\sum_{\overline{\eta}}\overline{\mu}(\overline{\eta})\sum_{y,z\in\Lambda\setminus\MC}\left(\grad_{y,z}\overline{f}(\overline{\eta})\right)^{2}\\
 & =\frac{CL}{k_{0}^{2}}\frac{\left|\overline{\Omega}_{k-N}\right|}{\left|\Omega_{k}\right|}\sum_{\eta}\mu(\eta)\One_{\MC\text{ is empty}}\sum_{y,z\in\Lambda\setminus\MC}\left(\grad_{y,z}f(\eta)\right)^{2}\\
 & \le\frac{CL}{k_{0}^{2}}\sum_{\eta}\mu(\eta)\One_{\MC\text{ is empty}}\sum_{y,z\in\Lambda\setminus\MC}\left(\grad_{y,z}f(\eta)\right)^{2}.
\end{align*}

In order to conclude we need to construct a multistep move that exchanges
$x$ and $y$:
\begin{claim}
Fix $y,z\in\Lambda\setminus\MC$. Then there exists a $T$-step move
$M_{y,z}=((\eta_{t}),(x_{t}),(e_{t}))$ such that:
\begin{enumerate}
\item $\Dom M_{y,z}=\left\{ \eta\in\Omega_{k}:\MC\text{ is empty}\right\} $.
\item $M_{y,z}$ is compatible with the transposition of $x$ and $y$.
\item $T\le CL$.
\item $\Loss M_{y,z}=0$.
\item Each site $x\in\Lambda$ is exchanged at most $C\boxsize^{C}$ times.
Moreover, 
\[
\left|\left\{ t\text{ such that }x_{t}(\eta)=x\text{ for some }\eta\in\Dom M_{s,z}\right\} \right|\le C\boxsize^{C}.
\]
\end{enumerate}
\end{claim}

\begin{proof}
If $y$ and $z$ are both larger than $\boxsize$, the construction
follows the exact same steps as that of $\overline{M}$ in the proof
of Lemma \ref{lem:flip_reservoir}.

If $y\in[\boxsize]$, we perform the following maneuver---first,
move the cluster $3\boxsize$ steps to the right. This move is compatible
with some permutation $\sigma$. Since the order of the particles
is conserved in one dimension, $\sigma(y)$ and $\sigma(y+\boxsize)$
are both in $[3\boxsize]$. We can then exchange them using the cluster
at $3\boxsize+\MC$ by the same construction as Lemma \ref{lem:flip_reservoir}.
When we now move the cluster back to the left, the net result is a
move compatible with transposing $y$ and $y+\boxsize$.

If $z>\boxsize$ we can apply the move constructed in the beginning,
exchaning $y+\boxsize$ with $z$, and finally wind back our manoeuvre
to exchange $y$ and $y+\boxsize$. This leaves us with the configuration
$\eta^{y,z}$ as we wanted.

If $z$ is also in $[l]$, we move the cluster $2\boxsize$ steps
to the right. Then use it to exchange $\sigma(y)$ and $\sigma(z)$.
Then move the cluster back $2\boxsize$ steps to the left.

If $L$ is large enough all these maneuvers take negligible time,
and we are left with the bound $T\le CL$.
\end{proof}
We can now use this newly constructed move $M_{y,z}=((\eta_{t}^{y,z}),(x_{t}^{y,z}),(e_{t}^{y,z}))$
in order to finish the bound on IV:
\begin{align*}
\text{IV} & \le\frac{CL^{2}}{k_{0}^{2}}\sum_{\eta}\mu(\eta)\One_{\MC\text{ is empty}}\sum_{y,z\in\Lambda\setminus\MC}\sum_{t=0}^{T-1}\sum_{\eta'}\sum_{x\in\Lambda}\One_{\eta'=\eta_{t}^{y,z}}\One_{x=x_{t}^{y,z}}c_{x,x+1}(\eta')\left(\grad_{x,x+1}f(\eta')\right)^{2}\\
 & =\frac{CL^{4}\boxsize^{C}}{k_{0}^{2}}\sum_{\eta'}\mu(\eta')\sum_{x\in\Lambda}c_{x,x+1}(\eta')\left(\grad_{x,x+1}f(\eta')\right)^{2}=\frac{CL^{4}\boxsize^{C}}{k_{0}^{2}}\Dirichlet f.
\end{align*}

To sum it all up, assuming $L$ is large enough and using the fact
that $k_{0}\ge q^{C}L$,
\[
\text{II}\le\text{III}+\text{IV}\le Cq^{-C}L\,\Dirichlet f+\frac{CL^{4}\boxsize^{C}}{k_{0}^{2}}\Dirichlet f\le Cq^{-C}\,L^{2}\,\Dirichlet f.
\]

We have thus proven inequalities (\ref{eq:inequality_I}) and (\ref{eq:inequality_II}),
concluding the proof of Theorem \ref{thm:gap_closed}. \qed

\section{\label{sec:diffusion}Diffusion coefficient}

In this section we consider the model on $\Z^{d}$, and study the
diffusion coefficient $D$. This is a symmetric matrix given by the
following variational formula (see, e.g., \cite[II.2.2]{Spohn2012IPS}):
for any $u\in\R^{d}$,
\begin{equation}
u\cdot Du=\frac{1}{2q(1-q)}\inf_{f}\,\mu\left[\sum_{\alpha=1}^{d}c_{0,e_{\alpha}}\left(u\cdot e_{\alpha}(\eta(0)-\eta(e_{\alpha}))+\sum_{x}\grad_{0,e_{\alpha}}\tau_{x}f\right)^{2}\right].\label{eq:D_variational}
\end{equation}

In \cite{GoncalvesLandimToninelli}, convergence to a hydrodynamic
limit of a variation of Example \ref{exa:1d} is proven, and the diffusion
coefficient is found explicitly. This is done by a careful choice
of the rates, rendering the model \emph{gradient}. Proving convergence
to a hydrodynamic limit for Example \ref{exa:1d} with the original
rates, and identifying the diffusion coefficient, is a much more difficult
task. However, equation (\ref{eq:D_variational}), together with the
result of \cite{GoncalvesLandimToninelli}, allows us to deduce the
positivity of the diffusion coefficient, and even give an estimate
accurate up to a factor (to be precise, $q\le D\le2q$). 

In this section we prove the positivity of the diffusion coefficient
in a much more general setting, for all noncooperative models.
\begin{thm}
\label{thm:D}Consider a noncooperative kinetically constrained lattice
gas, and let $D$ be the associated diffusion coefficient (given in
equation (\ref{eq:D_variational})). Then $D$ is positive definite,
that is, $u\cdot Du$ is strictly positive for any $u\in\R^{d}$.
\end{thm}

\begin{rem}
The proof of Theorem (\ref{thm:D}) also provides bounds on the diffusion
coefficient, and in particular shows that it could decay at most
polynomially fast as $q$ tends to $0$. This power law behavior is
characteristic of noncooperative models, while cooperative models
are expected to show faster decay (see e.g. \cite{S23KA_HL}).
\end{rem}

\subsection{Proof}

\subsubsection{Comparison argument}

We will see here how to bound the diffusion coefficient using multistep
moves that compare our model to an auxiliary dynamics. For this purpose,
consider the dynamics defined by a generator 
\begin{equation}
\cL_{\text{aux}}f=\sum_{x\sim y}c_{x,y}^{\text{aux}}(\eta)\grad_{x,y}f(\eta),\label{eq:generator_auxiliary}
\end{equation}
and assume:
\begin{enumerate}
\item The rates $c_{x,y}^{\text{aux}}(\eta)$ do not depend on $\eta(x)$,
$\eta(y)$. This guarantees that the dynamics is reversible with respect
to $\mu$.
\item The model is translation invariant.
\item The rates are bounded from above by $c_{\text{max}}^{\text{aux}}$.
\end{enumerate}
In order to compare the two models, we need to be able to perform
the exchanges of the auxiliary model using the original dynamics.
This will be done using a multistep move:
\begin{hypothesis}
\label{hyp:auxiliary_move}For any $\alpha\in\{1,\dots,d\}$ there
exists a $T_{\text{Aux}}$-step move $\text{Aux}_{\alpha}$ such that:
\begin{enumerate}
\item $\Dom(\text{Aux}_{\alpha})=\left\{ \eta\in\Omega:c_{0,e_{\alpha}}^{\text{aux}}(\eta)\neq0\right\} $,
\item The move is compatible with the permutation exchanging $0$ and $e_{\alpha}$.
\item $x_{t}\in\Lambda$ for all $t$, where $\Lambda$ is a fixed set.
\end{enumerate}
\end{hypothesis}

\begin{lem}
\label{lem:D_comparison}Consider the auxiliary model (\ref{eq:generator_auxiliary}),
and let $D^{\text{aux}}$ be its diffusion coefficient. If Hypothesis
(\ref{hyp:auxiliary_move}) is satisfied, then for any $u\in\R^{d}$
\[
u\cdot D^{\text{aux}}u\le dT_{\text{Aux}}^{2}2^{\Loss(\text{Aux})}c_{\text{max}}^{\text{aux}}\left|\Lambda\right|\,u\cdot Du.
\]
\end{lem}

\begin{proof}
Fix a local function $f:\Omega\to\R$. We need to show that
\begin{multline*}
\sum_{\alpha=1}^{d}\mu\left[c_{0,e_{\alpha}}^{\text{aux}}\left(u\cdot e_{\alpha}(\eta(0)-\eta(e_{\alpha}))+\sum_{x}\grad_{0,e_{\alpha}}\tau_{x}f\right)^{2}\right]\\
\le dT_{\text{Aux}}^{2}2^{\Loss(\text{Aux})}c_{\text{max}}^{\text{aux}}\left|\Lambda\right|\sum_{\alpha=1}^{d}\mu\left[c_{0,e_{\alpha}}\left(u\cdot e_{\alpha}(\eta(0)-\eta(e_{\alpha}))+\sum_{x}\grad_{0,e_{\alpha}}\tau_{x}f\right)^{2}\right].
\end{multline*}

Fix $\alpha$, and denote $\text{Aux}_{\alpha}=\left((\eta_{t}),(x_{t}),(e_{t})\right)$.
Then, for $\eta\in\Dom(\text{Aux}_{\alpha})$ we can write
\begin{align*}
u\cdot e_{\alpha}\left(\eta(0)-\eta(e_{\alpha})\right) & =\sum_{t=0}^{T-1}u\cdot e_{t}\,\left(\eta_{t}(x_{t})-\eta_{t}(x_{t}+e_{t})\right),\\
\grad_{0,e_{\alpha}}\tau_{x}f & =\sum_{t=0}^{T-1}\grad_{x_{t},x_{t}+e_{t}}\,\tau_{x}f(\eta_{t}).
\end{align*}

Using these equalities, 
\begin{multline*}
\mu\left[c_{0,e_{\alpha}}^{\text{aux}}\left(u\cdot e_{\alpha}(\eta(0)-\eta(e_{\alpha}))+\sum_{x}\grad_{0,e_{\alpha}}\tau_{x}f\right)^{2}\right]\\
=\mu\left[c_{0,e_{\alpha}}^{\text{aux}}\left(\sum_{t=0}^{T}u\cdot e_{t}\left(\eta_{t}(x_{t})-\eta_{t}(x_{t}+e_{t})\right)+\sum_{x}\sum_{t=0}^{T}\grad_{x_{t},x_{t}+e_{t}}\,\tau_{x}f\right)^{2}\right]\\
\le T_{\text{Aux}}\mu\left[c_{0,e_{\alpha}}^{\text{aux}}\sum_{t=0}^{T}\left(u\cdot e_{t}\left(\eta_{t}(x_{t})-\eta_{t}(x_{t}+e_{t})\right)+\sum_{x}\tau_{x_{t}}\grad_{0,e_{t}}\tau_{-x_{t}}\,\tau_{x}f\right)^{2}\right]\\
=T_{\text{Aux}}\mu\left[c_{0,e_{\alpha}}^{\text{aux}}\sum_{t=0}^{T}c_{x_{t},x_{t}+e_{t}}(\eta_{t})\left(u\cdot e_{t}\tau_{x_{t}}\left(\eta_{t}(0)-\eta_{t}(e_{t})\right)+\tau_{x_{t}}\sum_{x}\grad_{0,e_{t}}\tau_{x}f\right)^{2}\right]\\
=T_{\text{Aux}}\sum_{\eta}\mu(\eta)c_{0,e_{\alpha}}^{\text{aux}}\sum_{t=0}^{T}\sum_{z\in\Lambda}\One_{z=x_{t}}\sum_{\eta'}\One_{\eta'=\tau_{z}\eta_{t}}\sum_{\alpha'}\One_{e_{\alpha'}=e_{t}}c_{0,e_{\alpha'}}(\eta')\\
\times\left(u\cdot e_{\alpha'}\left(\eta'(0)-\eta'(e_{\alpha'})\right)+\sum_{x}\grad_{0,e_{\alpha'}}\tau_{x}f(\eta')\right)^{2}\\
=T_{\text{Aux}}^{2}2^{\Loss(\text{Aux})}c_{\text{max}}^{\text{aux}}\left|\Lambda\right|\sum_{\eta'}\mu(\eta')c_{0,e_{\alpha'}}(\eta')\sum_{\alpha'}\left(u\cdot e_{\alpha'}\left(\eta'(0)-\eta'(e_{\alpha'})\right)+\sum_{x}\grad_{0,e_{\alpha'}}\tau_{x}f(\eta')\right)^{2}\\
=T_{\text{Aux}}^{2}2^{\Loss(\text{Aux})}c_{\text{max}}^{\text{aux}}\left|\Lambda\right|\,\sum_{\alpha'=1}^{d}\mu\left[c_{0,e_{\alpha'}}\left(u\cdot e_{\alpha'}\left(\eta(0)-\eta(e_{\alpha'})\right)+\sum_{x}\grad_{0,e_{\alpha'}}\tau_{x}f\right)^{2}\right].
\end{multline*}
\end{proof}

\subsubsection{The auxiliary model}

We now define an auxiliary model that will satisfy Hypothesis \ref{hyp:auxiliary_move}.
In order to do that, fix $d$ finite sets of sites, $\cA^{\alpha}=\left\{ x_{1}^{\alpha},\dots,x_{n_{\alpha}}^{\alpha}\right\} $
for $\alpha\in\{1,\dots,d\}$. We order $x_{1}^{\alpha},\dots,x_{n_{\alpha}}^{\alpha}$
from right to left according to their $\alpha$ coordinate, so that
$x_{i}^{\alpha}\cdot e_{\alpha} \ge x_{j}^{\alpha}\cdot e_{\alpha}$ 
if $i\le j$. We also define the sets 
\[
\cA_{i}^{\alpha}=\left\{ x_{j}^{\alpha}+e_{\alpha}\,,\,1\le j\le i\right\} \cup\left\{ x_{j}^{\alpha}\,,\,i+1\le j\le n_{\alpha}\right\} 
\]
for $i\in\{0,\dots,n_{\alpha}\}$, so that $\cA_{0}^{\alpha}=\cA^{\alpha}$,
and $\cA_{i+1}^{\alpha}$ is obtained from $\cA_{i}^{\alpha}$ by
moving $x_{i+1}^{\alpha}$ one step in the direction $e_{\alpha}$.
Note that thanks to the ordering we have chosen, the new site $x_{i}^{\alpha}+e_{\alpha}$
does not belong to $\cA_{i}^{\alpha}$, so that $\left|\cA_{i}^{\alpha}\right|=n_{\alpha}$
for all $i$, and $\cA_{n_{\alpha}}^{\alpha}=\cA^{\alpha}+e_{\alpha}$.

We will now define a Markov process on $\Omega$ with the aid of these
sets. The idea would be to allow empty copies of $\cA^{\alpha}$ to
move in the direction $\pm e_{\alpha}$, vacancy by vacancy, by changing
at each step $\cA_{i}^{\alpha}$ to $\cA_{i\pm1}^{\alpha}$. More
precisely, for each $\alpha$ and each $i\in\{0,\dots,n_{\alpha}-1\}$,
we identify all translations of $\cA_{i}^{\alpha}$ of the form $x+\cA_{i}^{\alpha}$
which are empty for $\eta$. Then, with rate $1$, we exchange sites
$x+x_{i+1}^{\alpha}$ and $x+x_{i+1}^{\alpha}+e_{\alpha}$. In addition,
for each $\alpha$ and each $i\in\{1,\dots,n_{\alpha}\}$, we identify
all translations of $\cA_{i}^{\alpha}$ of the form $x+\cA_{i}^{\alpha}$
which are empty for $\eta$. Then, with rate $1$, we exchange sites
$x+x_{i}^{\alpha}$ and $x+x_{i}^{\alpha}+e_{\alpha}$. This could
be described using the following infinitesimal generator operating
on a local function $f$:
\begin{eqnarray}
\cL_{\text{aux}}f & = & \sum_{\alpha=1}^{d}\sum_{i=0}^{n_{\alpha}-1}\sum_{x\in\Z^{d}}\One_{x+\cA_{i}^{\alpha}\text{ are empty}}\grad_{x+x_{i+1}^{\alpha},x+x_{i+1}^{\alpha}+e_{\alpha}}f(\eta)\label{eq:auxiliary_dynamics}\\
 &  & +\sum_{\alpha=1}^{d}\sum_{i=1}^{n_{\alpha}}\sum_{x\in\Z^{d}}\One_{x+\cA_{i}^{\alpha}\text{ are empty}}\grad_{x+x_{i}^{\alpha},x+x_{i}^{\alpha}+e_{\alpha}}f(\eta).\nonumber 
\end{eqnarray}
We will refer to the transition described in the first sum as \emph{forward
transitions}, and to the ones in the second sum as \emph{backward
transitions}. That is, a forward transition occurs when an empty site
$x$ is exchanged with an occupied neighbor $x+e_{\alpha}$, and a
backward transition occurs when an empty site $y$ is exchanged with
an occupied neighbor $y-e_{\alpha}$. Note that a forward transition
from $x$ to $x+e_{\alpha}$ is only possible when for some $\tilde{x}\in\Z^{2}$
and $i\in\{0,\dots,n_{\alpha}-1\}$, $\tilde{x}+\cA_{i}^{\alpha}$
is empty and $x=\tilde{x}+x_{i+1}^{\alpha}$. In other words, we need
$x-x_{i+1}^{\alpha}+\cA_{i}^{\alpha}$ to be empty for some $i\in\{0,\dots,n_{\alpha}-1\}$.
Similarly, a backward transition from $y$ to $y-e_{\alpha}$ requires
$y-e_{\alpha}-x_{i}^{\alpha}+\cA_{i}^{\alpha}$ to be empty for some
$i\in\{1,\dots,n_{\alpha}\}$.
\begin{observation}
The auxiliary dynamics (\ref{eq:auxiliary_dynamics}) is reversible
with respect to the equilibrium measure $\mu$, for any value of the
parameter $q$.
\end{observation}

\begin{proof}
This is a consequence of the fact that for any $\eta\in\Omega$ and
any edge $x\sim y$ of $\Z^{2}$, the rate at which $\eta$ changes
to $\eta^{x,y}$ is the same as the rate at which $\eta^{x,y}$ changes
to $\eta$---without loss of generality assume $\eta(x)=1-\eta(y)=0$
and $y=x+e_{\alpha}$. Then the rate of exchanging $x$ and $y$ for
$\eta$ is given by the number of sets $\cA_{i}^{\alpha}$ , $i\in\{0,\dots,n_{\alpha}-1\}$,
such that $x-x_{i+1}+\cA_{i}^{\alpha}$ is empty for $\eta$. On the
other hand, the rate of exchanging $x$ and $y$ for $\eta^{x,y}$
is given by the number of sets $\cA_{i}^{\alpha}$, $i\in\{1,\dots,n_{\alpha}\}$,
such that $y-e_{\alpha}-x_{i}+\cA_{i}^{\alpha}$ is empty for $\eta^{x,y}$.
The latter could be written as 
\begin{gather*}
\#\left\{ i\in\{1,\dots,n_{\alpha}\}:y-e_{\alpha}-x_{i}+\cA_{i}^{\alpha}\text{ is empty for }\eta^{x,y}\right\} \\
=\#\left\{ i\in\{0,\dots,n_{\alpha}-1\}:x-x_{i+1}+\cA_{i+1}^{\alpha}\text{ is empty for }\eta^{x,y}\right\} \\
=\#\left\{ i\in\{0,\dots,n_{\alpha}-1\}:x-x_{i+1}+\cA_{i}^{\alpha}\text{ is empty for }\eta\right\} ,
\end{gather*}
which conclude the proof.
\end{proof}
The last observation shows that $\cL_{\text{aux}}$ could be put in
the form (\ref{eq:generator_auxiliary}), where the rates $c_{x,y}^{\text{aux}}$
are bounded by $n_{\alpha}$.

The key property of this model is that the total current vanishes
for any configuration:
\begin{observation}
\label{obs:zero_current}Consider the auxiliary dynamics (\ref{eq:auxiliary_dynamics})
on the torus $\Z^{d}/L\Z^{d}$, for some fixed (large) $L$. Then,
for any $\eta\in\Omega$, the total current is $0$. That is,
\[
\sum_{x\sim y}c_{x,y}^{\text{aux}}\,(x-y)\left(\eta(x)-\eta(y)\right)=0.
\]
\end{observation}

\begin{proof}
Fix $\alpha\in\{1,\dots,d\}$. We show that the total current in the
$\alpha$ direction is $0$. The negative current (particles moving
in the direction $-e_{\alpha}$) is given by forward transitions,
and the positive current by backward transitions. We need to prove
that the two cancel out.

Each empty translation of $\cA_{i}^{\alpha}$ contributes a forward
transition of rate $1$, unless we try to move the vacancy to an already
empty site. Hence the rate of forward transitions is given by 
\begin{multline*}
\sum_{i=0}^{n_{\alpha}-1}\sum_{x\in\Z^{d}}\One_{x+\cA_{i}^{\alpha}\text{ are empty}}-\sum_{i=0}^{n_{\alpha}-1}\sum_{x\in\Z^{d}}\One_{x+\cA_{i}^{\alpha}\text{ are empty}}\One_{x+x_{i+1}^{\alpha}+e_{\alpha}\text{ is empty}}\\
=\sum_{i=1}^{n_{\alpha}-1}\sum_{x\in\Z^{d}}\One_{x+\cA_{i}^{\alpha}\text{ are empty}}+\sum_{x\in\Z^{d}}\One_{x+\cA_{0}^{\alpha}\text{ are empty}}-\sum_{i=0}^{n_{\alpha}-1}\sum_{x\in\Z^{d}}\One_{x+\cA_{i+1}^{\alpha}\text{ are empty}}\One_{x+x_{i+1}^{\alpha}\text{ is empty}}\\
=\sum_{i=1}^{n_{\alpha}-1}\sum_{x\in\Z^{d}}\One_{x+\cA_{i}^{\alpha}\text{ are empty}}+\sum_{x\in\Z^{d}}\One_{x+\cA_{n_{\alpha}}^{\alpha}\text{ are empty}}-\sum_{i=1}^{n_{\alpha}}\sum_{x\in\Z^{d}}\One_{x+\cA_{i}^{\alpha}\text{ are empty}}\One_{x+x_{i}^{\alpha}\text{ is empty}}\\
=\sum_{i=1}^{n_{\alpha}}\sum_{x\in\Z^{d}}\One_{x+\cA_{i}^{\alpha}\text{ are empty}}-\sum_{i=1}^{n_{\alpha}}\sum_{x\in\Z^{d}}\One_{x+\cA_{i}^{\alpha}\text{ are empty}}\One_{x+x_{i}^{\alpha}\text{ is empty}}.
\end{multline*}
We recognize the last line as the rate of backward transitions, which
finishes the proof.
\end{proof}
The zero current property, as explained in \cite[II.2.4]{Spohn2012IPS},
makes the contribution of the current-current correlation to the diffusion
coefficient vanish. This allows us to calculate explicitly the diffusion
coefficient.
\begin{lem}
\label{lem:D_aux}Let $D^{\text{aux}}$ be the diffusion coefficient
associated to the auxiliary dynamics (\ref{eq:auxiliary_dynamics}).
Then for any $u\in\R^{d}$
\begin{align*}
u\cdot D^{\text{aux}}u & =\sum_{\alpha=1}^{d}(u\cdot e_{\alpha})^{2}\,\mu\left[c_{0,e_{\alpha}}\right]\ge Cq^{n}\,\norm u^{2},
\end{align*}
where $n=\max_{\alpha}n_{\alpha}$.
\end{lem}

\begin{proof}
The inequality follows directly from the definition of the model,
so we are left with showing the equality. \cite[II.2.4]{Spohn2012IPS}
explains how it could be derived from the Green-Kubo formula \cite[II, equation (2.27)]{Spohn2012IPS},
for completeness we will prove it explicitly from the variational
characterization (\ref{eq:D_variational}).

Fix a local function $f$, and $L$ large enough (depending on the
support of $f$), so that $\sum_{x\in\Z^{d}}$ in equation (\ref{eq:D_variational})
could be replaced by $\sum_{x\in\Z^{d}/L\Z^{d}}$. Then 
\begin{multline*}
\mu\left[\sum_{\alpha=1}^{d}c_{0,e_{\alpha}}^{\text{aux}}\left(u\cdot e_{\alpha}(\eta(0)-\eta(e_{\alpha}))+\sum_{x}\grad_{0,e_{\alpha}}\tau_{x}f\right)^{2}\right]\\
=\sum_{\alpha=1}^{d}\mu\left[c_{0,e_{\alpha}}^{\text{aux}}\left(u\cdot e_{\alpha}(\eta(0)-\eta(e_{\alpha}))\right)^{2}\right]+2\sum_{\alpha=1}^{d}\mu\left[c_{0,e_{\alpha}}^{\text{aux}}\,u\cdot e_{\alpha}(\eta(0)-\eta(e_{\alpha}))\,\sum_{x}\grad_{0,e_{\alpha}}\tau_{x}f\right]\\
+\sum_{\alpha=1}^{d}\mu\left[c_{0,e_{\alpha}}^{\text{aux}}\left(\sum_{x}\grad_{0,e_{\alpha}}\tau_{x}f\right)^{2}\right]\\
\ge\sum_{\alpha=1}^{d}\mu\left[c_{0,e_{\alpha}}^{\text{aux}}\left(u\cdot e_{\alpha}(\eta(0)-\eta(e_{\alpha}))\right)^{2}\right]+2\sum_{\alpha=1}^{d}u\cdot e_{\alpha}\sum_{x}\mu\left[c_{0,e_{\alpha}}^{\text{aux}}\,(\eta(0)-\eta(e_{\alpha}))\,\grad_{0,e_{\alpha}}\tau_{x}f\right].
\end{multline*}
Since $\mu$ is invariant under the map $\eta\mapsto\eta^{0,e_{\alpha}}$
and $c_{0,e_{\alpha}}^{\text{aux}}(\eta)=c_{0,e_{\alpha}}^{\text{aux}}(\eta^{0,e_{\alpha}})$,
we can write for any function $g$
\begin{align*}
\mu\left[c_{0,e_{\alpha}}^{\text{aux}}\,(\eta(0)-\eta(e_{\alpha}))\,g(\eta)\right] & =\frac{1}{2}\left(\mu\left[c_{0,e_{\alpha}}^{\text{aux}}\,(\eta(0)-\eta(e_{\alpha}))\,g(\eta)\right]+\mu\left[c_{0,e_{\alpha}}^{\text{aux}}\,(\eta^{0,e_{\alpha}}(0)-\eta^{0,e_{\alpha}}(e_{\alpha}))\,g(\eta^{0,e_{\alpha}})\right]\right)\\
 & =-\frac{1}{2}\mu\left[c_{0,e_{\alpha}}^{\text{aux}}\,(\eta(0)-\eta(e_{\alpha}))\,\grad_{0,e_{\alpha}}g(\eta)\right].
\end{align*}
Therefore, setting $g=\tau_{x}f$ and then using the translation invariance
of $\mu$ we obtain
\begin{multline*}
\sum_{x}\mu\left[c_{0,e_{\alpha}}^{\text{aux}}\,(\eta(0)-\eta(e_{\alpha}))\,\grad_{0,e_{\alpha}}\tau_{x}f\right]=-2\sum_{x}\mu\left[c_{0,e_{\alpha}}^{\text{aux}}\,(\eta(0)-\eta(e_{\alpha}))\,\tau_{x}f\right]\\
=-2\sum_{x}\mu\left[c_{x,x+e_{\alpha}}^{\text{aux}}\,(\eta(x)-\eta(x+e_{\alpha}))\,f\right]=-2\mu\left[\left(\sum_{x}c_{x,x+e_{\alpha}}^{\text{aux}}\left(\eta(x)-\eta(x+e_{\alpha})\right)\right)f\right].
\end{multline*}
The last term is $0$ by Observation (\ref{obs:zero_current}), proving
that 
\[
u\cdot D^{\text{aux}}u\ge\frac{1}{2q(1-q)}\,\mu\left[\sum_{\alpha=1}^{d}c_{0,e_{\alpha}}^{\text{aux}}\left(u\cdot e_{\alpha}(\eta(0)-\eta(e_{\alpha}))\right)^{2}\right].
\]
Hence, the infimum in equation (\ref{eq:D_variational}) is attained
for constant $f$. 

Finally, we use the product structure of $\mu$
and the fact that $c_{0,e_{\alpha}}^{\text{aux}}$ does not depend
on $\eta(0)$ and $\eta(e_{\alpha})$ to calculate this infimum explicitly:
\begin{align*}
u\cdot D^{\text{aux}}u & =\frac{1}{2q(1-q)}\,\mu\left[\sum_{\alpha=1}^{d}c_{0,e_{\alpha}}\left(u\cdot e_{\alpha}(\eta(0)-\eta(e_{\alpha}))\right)^{2}\right]\\
 & =\frac{1}{2q(1-q)}\sum_{\alpha=1}^{d}(u\cdot e_{\alpha})^{2}\mu\left[c_{0,e_{\alpha}}\right]\left[\left(\eta(0)-\eta(e_{\alpha})\right)^{2}\right]\\
 & =\sum_{\alpha=1}^{d}(u\cdot e_{\alpha})^{2}\mu\left[c_{0,e_{\alpha}}\right]. \qedhere
\end{align*}
\end{proof}

\subsubsection{The multistep move}

As a corollary of lemmas \ref{lem:D_comparison} and \ref{lem:D_aux},
if we assume that for any $\alpha$ there exists $\cA^{\alpha}$ of
size $n_{\alpha}\le n$ such that the auxiliary model defined in (\ref{eq:auxiliary_dynamics})
satisfies Hypothesis (\ref{hyp:auxiliary_move}), then 
\[
u\cdot Du\ge Cq^{n}\,\norm u^{2}
\]
for any $u\in\R^{d}$.
\begin{example}
In Example \ref{exa:1d}, we may take $\cA_{0}=\{1,2\}$ so $\cA_{1}=\{1,3\}$
and $\cA_{2}=\{2,3\}$. Then the multistep $\text{Aux}$ could be
chosen trivially as the $1$-step move exchanging the corresponding
sites.

Similarly, in Example \ref{exa:1d}, we take $\cA^{1}=\{e_{1},2e_{1}\}$
and $\cA^{2}=\{e_{2},2e_{2}\}$, and verify that we may choose the
trivial $1$-step moves.

In these two examples we know that by modifying the rates (without
changing the constrained and unconstrained transitions) as in \cite{GoncalvesLandimToninelli}
we obtain a gradient model (which is, in fact, the auxiliary model
we defined above). That is, equation (\ref{eq:D_variational}) could be used
directly, without passing through the comparison argument. This is
expressed in the fact that our multistep move is in fact a $1$-step
move.
\end{example}

In order to prove Theorem \ref{thm:D} all that is left is to construct
$\cA^{\alpha}$ and the $\text{Aux}_{\alpha}$ move. Consider a mobile
cluster $\MC$, and $\MClen$ such that $\MC\in[\MClen-1]^{d}$. Choosing,
for any $\alpha$, the set $\cA^{\alpha}=\MC\cup\left(\MClen e_{\alpha}+\MC\right)$
(with $n_{\alpha}=2\left|\MC\right|$) will suffice. In order to show
that, we need to construct the $\text{Aux}_{\alpha}$ move.

Let $\eta\in\Dom\text{Aux}_{\alpha}$, i.e., $c_{0,e_{\alpha}}^{\text{aux}}>0$.
By reversibility we may assume that this is a forward transition,
so $\eta(0)=1-\eta(e_{\alpha})=0$, and there exists $i\in\{0,\dots,n_{\alpha}-1\}$
such that $-x_{i+1}+\cA_{i}^{\alpha}$ is empty. We consider two cases:
\begin{casenv}
\item $i\in\{0,\dots,\left|\MC\right|-1\}$. Then $-x_{i+1}+\MC=-x_{i+1}+\{x_{\left|\MC\right|+1},\dots,x_{n_{\alpha}}\}\subset\cA_{i}^{\alpha}$.
Moreover, neither $0$ nor $e_{\alpha}$ are contained in $-x_{i+1}+\MC$
since $x_{i+1}\in\MClen e_{\alpha}+[\MClen-1]^{d}$. We may therefore
apply translation and exchange moves using the mobile cluster $-x_{i+1}+e_{\alpha}+\MClen e_{\alpha}+\MC$
in order to exchange $0$ and $e_{\alpha}$.
\item $i\in\{\left|\MC\right|,\dots,n_{\alpha}\}$. Then $-x_{i+1}+e_{\alpha}+\MClen e_{\alpha}+\MC=-x_{i+1}+e_{\alpha}+\{x_{1},\dots,x_{\left|\MC\right|}\}\subset\cA_{i}^{\alpha}$.
As before, neither $0$ nor $e_{\alpha}$ are contained in $-x_{i+1}+e_{\alpha}+\MClen e_{\alpha}+\MC$
since $x_{i+1}\in[\MClen-1]^{d}$. We may therefore apply translation
and exchange moves using the mobile cluster $-x_{i+1}+e_{\alpha}+\MClen e_{\alpha}+\MC$
in order to exchange $0$ and $e_{\alpha}$.
\end{casenv}
Hypothesis \ref{hyp:auxiliary_move} is thus satisfied, concluding
the proof of Theorem \ref{thm:D} by lemmas \ref{lem:D_comparison}
and \ref{lem:D_aux}. \qed
\begin{rem}
While the construction above gives a polynomial bound for all noncooperative
models, in specific cases it might not be optimal. In Example \ref{exa:2d},
the mobile cluster has size $4$, therefore the estimate we obtain
is of the order $q^{8}$. We have seen, however, that there is a more
efficient explicit choice of $\cA^{\alpha}$ which yields a much better
bound, of the order $q^{2}$.
\end{rem}

\section{\label{sec:self_diffusion}Self-diffusion in $d\ge2$}

In this section we study the self-diffusion coefficient $D_{s}$,
which is a symmetric matrix given by the following variational formula
(\cite{Spohn90SelfDiff}, \cite[II.6.2]{Spohn2012IPS}): for any $u\in\R^{d}$,

\begin{equation}
u\cdot D_{s}u=\frac{1}{2}\inf_{f}\mu_{0}\left[\sum_{\substack{y\sim x\\
x,y\neq0
}
}c_{xy}(\grad_{xy}f)^{2}+\sum_{y\sim0}c_{0y}(1-\eta(y))\left(u\cdot y+f(\tau_{-y}\eta^{0y})-f(\eta)\right)^{2}\right].\label{eq:D_s}
\end{equation}

In dimension $1$, due to the preservation of the order of particles,
the self-diffusion coefficient is $0$ even with in an unconstrained
setting (see, e.g., \cite[II.6.4]{Spohn2012IPS}), we will therefore
consider here only the higher dimensional case.

The positivity of the diffusion coefficient for examples \ref{exa:1d}
and \ref{exa:2d} was proven in \cite{BT04}. We will see here that
it is positive for any noncooperative models.
\begin{thm}
\label{thm:self_diffusion}Consider a noncooperative kinetically constrained
lattice gas in dimension $2$ or higher, and let $D_{s}$ be the associated
self-diffusion coefficient (given in equation (\ref{eq:D_s})). Then
$D_{s}$ is positive definite, that is, $u\cdot D_{s}u$ is strictly
positive for any $u\in\R^{d}$.
\end{thm}

\begin{rem}
As for the diffusion coefficient, the proof of Theorem \ref{thm:self_diffusion}
also shows that the rate at which $D_{s}$ decays to $0$ when $q$
approaches $0$ is at most polynomial, as expected.
\end{rem}

\subsection{Proof}

The proof will follow the strategy of \cite[II.6.3]{Spohn2012IPS},
also used in \cite{BT04}. It consists of comparing the model to as
auxiliary model where the tracer motion could be more easily tracked.
The auxiliary model we will choose, however, does not fall under
the framework of equation (\ref{eq:D_s})---First, the transitions
are not single particle jumps, but a simultaneous rearrangement of
several particles. Moreover, these transitions are not homogeneous;
more precisely, the allowed transitions and their rates depend on
the position as seen from the tracer.

We start by generalizing equation (\ref{eq:D_s}) in a setting which
will cover our auxiliary model. Consider a dynamics on the space of
configuration $\Omega$ with additional information on the location
of the tracer $z\in\Z^{d}$. Fix a countable set $\Sigma$ of permutations
of the sites, and assume that they all have finite range. This means
that, for some fixed $R$, any permutation $\sigma\in\Sigma$ fixes
the sites outside $x+[-R,R]^{d}$, where $x\in\Z^{d}$ may depend
on $\sigma$. Then, for each $\sigma\in\Sigma$, we apply $\sigma$
with rate $\hat{c}_{\sigma}$, relative to the tracer position $z$.
That is, the configuration $\eta$ becomes $\tau_{z}\sigma\tau_{-z}\eta$
and the tracer moves to $\tau_{z}\sigma\tau_{-z}(z)=z+\sigma(0)$,
with rate $\hat{c}_{\sigma}(\tau_{-z}\eta)$. It is important to note
that in the new configuration, if the old tracer position is occupied
then so is the new one. This process can be written using the infinitesimal
generator operating on $f:\Z^{d}\times\Omega\to\R$:
\begin{equation}
\hat{\cL}f(z,\eta)=\sum_{\sigma\in\Sigma}\hat{c}_{\sigma}(\tau_{-z}\eta)\left(f(z+\sigma(0),\tau_{z}\sigma\tau_{-z}\eta)-f(z,\eta)\right),\label{eq:tracer_generator}
\end{equation}
for a set of rates $\hat{c}_{\sigma}:\Omega\to[0,\infty)$ defined
for all any $\sigma\in\Sigma$.
\begin{rem}
\label{rem:tracer_gereralization}To obtain the original kinetically
constrained model we take $\Sigma$ to be the set of nearest neighbor
transpositions $\Sigma_{\text{kc}}$, and the rate 
\[
\hat{c}_{(x,y)}^{\text{kc}}(\eta)=\begin{cases}
c_{x,y}(\tau_{-z}\eta)\One_{\eta(y)=0} & \text{if }x=0,\\
c_{x,y}(\tau_{-z}\eta)\One_{\eta(x)=0} & \text{if }y=0,\\
c_{x,y}(\tau_{-z}\eta) & \text{otherwise}.
\end{cases}
\]
The reason that we do not simply take $\hat{c}_{(x,y)}^{\text{kc}}(\eta)=c_{x,y}(\tau_{-z}\eta)$
is that, while in the original dynamics exchanging two particles is
equivalent to doing nothing, when following the tracer we are not
allowed to exchange it with a particle.

Then
\begin{eqnarray*}
\hat{\cL}^{\text{kc}} f(z,\eta) & = & \sum_{\substack{x\sim y\\
x,y\neq0
}
}c_{x,y}(\tau_{-z}\eta)\left(f(z,\eta^{x+z,y+z})-f(z,\eta)\right)+\sum_{0\sim y}c_{0,y}(\tau_{-z}\eta)\left(f(y,\eta^{z,y+z})-f(z,\eta)\right)\\
 & = & \sum_{\substack{x\sim y\\
x,y\neq z
}
}c_{x,y}(\eta)\left(f(z,\eta^{x,y})-f(z,\eta)\right)+\sum_{z\sim y}c_{z,y}(\eta)\left(f(y,\eta^{z,y})-f(z,\eta)\right),
\end{eqnarray*}
which is indeed the generator of the dynamics (\ref{eq:generator})
together with a tracer.
\end{rem}

The variational formula (\ref{eq:D_s}) could be generalized to the
setting of (\ref{eq:tracer_generator}):
\begin{lem}
\label{lem:Ds_generalized}Consider the dynamics (\ref{eq:tracer_generator}).
Assume that, ignoring the tracer, it is reversible with respect to
a probability measure $\nu$ on $\Omega$ (i.e., $\hat{\cL}$ is self
adjoint operating on functions that do not depend on $z$). Let $\nu_{0}$
be the measure $\nu$, conditioned on having a particle at the origin,
i.e., $\nu_{0}(\zeta\in\cdot)=\nu(\zeta\in\cdot|\zeta(0)=1)$. Then
for any $u\in\R^{d}$, 
\[
u\cdot\hat{D}_{s}u=\frac{1}{2}\inf_{f}\left\{ \sum_{\sigma\in\Sigma}\nu_{0}\left[\hat{c}_{\sigma}(\zeta)\left(u\cdot\sigma(0)+f(\tau_{-\sigma(0)}\sigma\zeta)-f(\zeta)\right)^{2}\right]\right\} ,
\]
where $\hat{D}_{s}$ is the associated self-diffusion coefficient
and the infimum is taken over all local functions on $\Omega_{0}=\{\zeta\in\Omega:\zeta(0)=1\}$.
\end{lem}

\begin{rem}
\label{rem:self_diffusion_generalization}From the last lemma we can
reconstruct equation (\ref{eq:D_s}): as in Remark \ref{rem:tracer_gereralization},
\begin{multline*}
\sum_{\sigma\in\Sigma}\nu_{0}\left[\hat{c}_{\sigma}^{\text{kc}}(\zeta)\left(f(\tau_{-\sigma(0)}\sigma\zeta)-f(\zeta)-u\cdot\sigma(0)\right)^{2}\right]\\
=\sum_{\substack{x\sim y\\
x,y\neq0
}
}\nu_{0}\left[c_{x,y}(\zeta)\left(f(\zeta^{x,y})-f(\zeta)\right)^{2}\right]+\sum_{y\sim0}\nu_{0}\left[c_{0,y}(\zeta)(1-\eta(y))\left(u\cdot y+f(\tau_{-y}\zeta^{0,y})-f(\zeta)\right)^{2}\right].
\end{multline*}
\end{rem}

\begin{proof}
The proof follows the exact same argument as \cite{Spohn90SelfDiff,Spohn2012IPS}.
For completeness we present here the main steps.

Consider the process described above, with $\eta_{t}$ and $z_{t}$
the configuration and tracer position at time $t$. Define $\zeta_{t}=\tau_{-z}\eta_{t}$,
so the joint process $(\zeta_{t},z_{t})$ is Markovian with generator
operating on $f:\Omega_{0}\times\Z^{d}\to\R$ as
\[
\overline{\cL}f(\zeta,z)=\sum_{\sigma\in\Sigma}\hat{c}_{\sigma}(\zeta)\left(f(z+\sigma(0),\tau_{-\sigma(0)}\sigma\zeta)-f(z,\zeta)\right).
\]

Fix $g(z,\zeta)=u\cdot z$, and let
\[
j_{u}(\zeta)=\overline{\cL}g(z,\zeta)=\sum_{\sigma\in\Sigma}\hat{c}_{\sigma}(\zeta)\,u\cdot\sigma(0).
\]
Then 
\[
u\cdot z_{t}-\int_{0}^{t}j_{u}(\zeta_{s})\dd s=M_{t}
\]
is a martingale with stationary increments and quadratic variation
\[
\E\left(M_{t}^{2}\right)=t\sum_{\sigma\in\Sigma}(u\cdot\sigma(0))^{2}\,\nu_{0}\left(\hat{c}_{\sigma}(\zeta)\right).
\]
Here, and in the rest of the proof, $\E(\cdot)$ refers to expectation
related to the process, starting from a configuration $\eta$
drawn according to $\nu_{0}$ and a tracer at the origin.

We obtain
\[
\E\left[(u\cdot z_{t})^{2}\right]=t\sum_{\sigma\in\Sigma}(u\cdot\sigma(0))^{2}\,\nu_{0}\left(\hat{c}_{\sigma}\right)-\int_{0}^{t}\int_{0}^{t}\E\left[j_{u}(\zeta_{s})j_{u}(\zeta_{s'})\right]\dd s\dd s'+\E\left[u\cdot z_{t}\,\int_{0}^{t}j_{u}(\zeta_{s})\dd s\right].
\]
By reversibility and translation invariance, the process $(-z_{t-s},\zeta_{t-s})_{s\in[0,t]}$
has the same law as $(z_{s},\zeta_{s})_{s\in[0,t]}$ (under the initial
condition $z=0$ and $\zeta$ draws from $\nu_{0}$). Therefore, the
last term in the equation above vanishes, leaving us with 
\[
u\cdot\hat{D}_{s}u=\frac{1}{2t}\lim_{t\to\infty}\E\left[(u\cdot z_{t})^{2}\right]=\frac{1}{2}\sum_{\sigma\in\Sigma}(u\cdot\sigma(0))^{2}\,\nu_{0}\left(\hat{c}_{\sigma}\right)-\int_{0}^{\infty}\nu_{0}\left[j_{u}e^{t\overline{\cL}}j_{u}\right]\dd t.
\]

Note that the last expression contains only functions of the configuration
$\zeta$, without looking at the tracer position $z$. The process
$(\zeta_{t})_{t=0}^{\infty}$ is Markovian and reversible with respect
to $\nu_{0}$; therefore, with some abuse of notation, we will consider
from now on $\overline{\cL}$ as the generator of this projected process,
operating on functions on $\Omega_{0}$.

We may now write
\begin{align*}
-\int\limits _{0}^{\infty}\nu_{0}\left[j_{u}e^{t\overline{\cL}}j_{u}\dd t\right] & =\nu_{0}\left[j_{u}\overline{\cL}^{-1}j_{u}\right]\\
 & =\inf_{f}\left\{ -2\nu_{0}(j_{u}f)-\nu_{0}(f\overline{\cL}f)\right\} .
\end{align*}

In order to calculate the first term in the infimum we use the detailed
balance equation. For every $\sigma$, defining $\sigma'=\tau_{-\sigma(0)}\sigma^{-1}\tau_{\sigma(0)}$
(so that applying $\sigma$ and then $\sigma'$ brings us back to
the original configuration),
\[
\nu_{0}\left[\hat{c}_{\sigma}(\zeta)f(\zeta)\right]=\nu_{0}\left[\hat{c}_{\sigma'}(\zeta)f(\tau_{-\sigma'(0)}\sigma'\zeta)\right].
\]
Hence, using $\sigma'(0)=-\sigma(0)$,
\begin{align*}
-2\nu_{0}\left[j_{u}f\right] & =-2\sum_{\sigma\in\Sigma}u\cdot\sigma(0)\,\nu_{0}\left[\hat{c}_{\sigma}(\zeta)f(\zeta)\right]\\
 & =\sum_{\sigma\in\Sigma}u\cdot\sigma(0)\,\nu_{0}\left[\hat{c}_{\sigma}(\zeta)\left(f(\tau_{-\sigma(0)}\sigma\zeta)-f(\zeta)\right)\right].
\end{align*}
The second term in the infimum is given by the Dirichlet form
\[
-\nu_{0}(f\overline{\cL}f)=\frac{1}{2}\sum_{\sigma\in\Sigma}\nu_{0}\left[\hat{c}_{\sigma}(\zeta)\left(f(\tau_{-\sigma(0)}\sigma\zeta)-f(\zeta)\right)^{2}\right].
\]

Summing all up,
\begin{multline*}
\frac{1}{2}\inf_{f}\left\{ \sum_{\sigma\in\Sigma}\nu_{0}\left[\hat{c}_{\sigma}(\zeta)\left(u\cdot\sigma(0)+f(\tau_{-\sigma(0)}\sigma\zeta)-f(\zeta)\right)^{2}\right]\right\} \\
=\frac{1}{2}\sum_{\sigma}(u\cdot\sigma(0))^{2}\nu_{0}(\hat{c}_{\sigma})+\inf_{f}\left\{ -\nu_{0}(f\overline{\cL}f)-2\nu_{0}(j_{u}f)\right\} \\
=\frac{1}{2}\sum_{\sigma}(u\cdot\sigma(0))^{2}\nu_{0}(\hat{c}_{\sigma})-\int\limits _{0}^{\infty}\nu_{0}\left[j_{u}e^{t\overline{\cL}}j_{u}\dd t\right]=u\cdot\hat{D}_{s}u. \qedhere
\end{multline*}
\end{proof}

\subsubsection{Comparison argument}

As in the case of the diffusion coefficient, we will see that an appropriate
move could help us compare different dynamics.

Consider a model as in equation (\ref{eq:tracer_generator}), satisfying
the following conditions:
\begin{enumerate}
\item For any $\sigma\in\Sigma$, the configuration $\sigma'=\tau_{-\sigma(0)}\sigma^{-1}\tau_{\sigma(0)}$
is also in $\Sigma$, and $\hat{c}_{\sigma}=\hat{c}_{\sigma'}$. This
is equivalent to reversibility with respect to the equilibrium measure
$\mu$ (for any $q$).
\item $\hat{c}_{\sigma}\le1$ for any $\sigma\in\Sigma$.
\end{enumerate}
The comparison argument will be based on multistep moves, requiring us 
to follow the tracer position throughout the move.
\begin{defn}
Fix a $T$-step move $M=((\eta_{t}),(x_{t}),(e_{t}))$, and assume
that for any $\eta\in\Dom(M)$ some given site $z_{0}$ is occupied,
i.e., $\eta(z_{0})=1$. Then the \emph{tracer position associated
with $M$ starting at $z_{0}$} is a sequence of sites $(z_{t})_{t=0}^{T}$
giving at each step $t$ the position of the particle originally at
\textbf{$z_{0}$}:
\[
z_{t+1}=\begin{cases}
x_{t}+e_{t} & \text{if }z_{t}=x_{t}\text{ and }\eta_{t}(x_{t}+e_{t})=0,\\
x_{t} & \text{if }z_{t}=x_{t}+e_{t}\text{ and }\eta_{t}(x_{t})=0,\\
z_{t} & \text{otherwise}.
\end{cases}
\]
\end{defn}

In order to compare the auxiliary model with our kinetically constrained
lattice gas, we must have an appropriate multistep move:
\begin{hypothesis}
\label{hyp:self_diffusion}For any $\sigma\in\Sigma$ and $z_{0}\in\Z^{d}$,
there is a $T$-step move $M_{z_{0},\sigma}=((\eta_{t}),(x_{t}),(e_{t}))$
such that:
\begin{enumerate}
\item $\text{Dom}M=\left\{ \eta\in\Omega:\eta(z_{0})=1\text{ and }\hat{c}_{\sigma}(\tau_{-z_{0}}\eta)>0\right\} $.
\item $M$ is compatible with the permutation $\tau_{z_{0}}\sigma\tau_{-z_{0}}$.
\item In all transitions involving the tracer, the site it jumps to must
be empty. More precisely, denote $z_{t}$ the tracer position associated
with $M$ starting from $z_{0}$. Then, for all $t$, if $x_{t}=z_{t}$
then $\eta_{t}(x_{t}+e_{t})=0$ and if $x_{t}+e_{t}=z_{t}$ then $\eta_{t}(x_{t})=0$.
\item For any $z_{0}$, $t$, $\eta'$, $x'$, $e'$ and $z'$,
\[
\left|\left\{ \sigma\in\Sigma:\eta_{t}=\eta',x_{t}=x',e_{t}=e',z_{t}=z'\right\} \right|\le C.
\]
\end{enumerate}
We note that by translation invariance of the kinetically constrained
lattice gas, it suffices to construct $M_{z_{0},\sigma}$ for a specific
choice of $z_{0}$ to guarantee its existence for all $z_{0}$.
\end{hypothesis}

\begin{lem}
\label{lem:Ds_comparison}Consider an auxiliary model as in (\ref{eq:tracer_generator}),
reversible with respect to $\mu$ and with rates bounded by $1$.
Assume that Hypothesis \ref{hyp:self_diffusion} holds. Then for all
$u\in\R^{d}$,
\[
u\cdot\hat{D}_{s}u\le C\,u\cdot D_{s}u,
\]
where $D_{s}$ and $\hat{D}_{s}$ are the self diffusion coefficients
associated with the kinetically constrained lattice gas and the auxiliary
model respectively.
\end{lem}

\begin{proof}
Fix $z_{0}\in\Z^{d}$ and $\sigma\in\Sigma$, and consider the move
$M_{z_{0},\sigma}=((\eta_{t}),(x_{t}),(e_{t}))$ given in Hypothesis
\ref{hyp:self_diffusion}. Let $z_{t}$ be the associated tracer position
starting at $z_{0}$. Fix $\eta\in\Dom M_{z_{0},\sigma}$, and set
$\zeta=\tau_{-z_{0}}\eta$, $\zeta_{t}=\tau_{-z}\eta_{t}$ and $\sigma_{t}=(x_{t}-z_{t},x_{t}-z_{t}+e_{t})$
for all $t$. Note first that
\begin{align*}
u\cdot\sigma(0)+f(\tau_{-\sigma(0)}\sigma\zeta)-f(\zeta) & =u\cdot(z_{T}-z_{0})+f(\zeta_{T})-f(\zeta_{0})\\
 & =\sum_{t=0}^{T-1}u\cdot(z_{t+1}-z_{t})+f(\zeta_{t+1})-f(\zeta_{t}).
\end{align*}
Also, 
\begin{align*}
z_{t+1} & =z_{t}+\sigma_{t}(0),\\
\zeta_{t+1} & =\tau_{-\sigma_{t}(0)}\sigma_{t}\zeta_{t}.
\end{align*}

Recall remarks \ref{rem:tracer_gereralization} and \ref{rem:self_diffusion_generalization}.
Setting $z_{0}=0$ (and hence $\zeta=\eta$),
\begin{multline*}
\sum_{\sigma\in\Sigma}\mu_{0}\left[\hat{c}_{\sigma}(\zeta)\left(\sum_{t=0}^{T-1}u\cdot(z_{t+1}-z_{t})+f(\zeta_{t+1})-f(\zeta_{t})\right)^{2}\right]\\
\le T\sum_{\sigma\in\Sigma}\mu_{0}\left[\hat{c}_{\sigma}(\zeta)\sum_{t=0}^{T-1}\left(u\cdot\sigma_{t}(0)+f(\tau_{-\sigma_{t}(0)}\sigma_{t}\zeta_{t})-f(\zeta_{t})\right)^{2}\right]\\
\le CT\sum_{z\in[-R,R]^{d}}\sum_{t=0}^{T-1}\mu_{0}\left[\sum_{\sigma'\in\Sigma_{\text{kc}}}\One_{z'=z_{t}}\One_{\sigma'=((x_{t}-z',x_{t}-z'+e_{t}))}\hat{c}_{\sigma'}^{\text{kc}}\left(u\cdot\sigma'(0)+f(\tau_{-\sigma'(0)}\sigma'\zeta')-f(\zeta')\right)^{2}\right]\\
\le CT^{2}R^{d}\mu_{0}\left[\sum_{\sigma'\in\Sigma_{\text{kc}}}\hat{c}_{\sigma'}^{\text{kc}}\left(u\cdot\sigma'(0)+f(\tau_{-\sigma'(0)}\sigma'\zeta')-f(\zeta')\right)^{2}\right].
\end{multline*}

This concludes the proof by Lemma \ref{lem:Ds_generalized}.
\end{proof}

\subsubsection{The auxiliary model}

Fix some finite set $\hat{\MC}\subset\Z^{d}\setminus\{0\}$, and $d$
permutations $\sigma_{1},\dots,\sigma_{d}$ with finite range. Assume
that $\sigma_{i}(0)=e_{i}$ and that $\sigma_{i}(\hat{\MC})=e_{i}+\hat{\MC}$.
For all $i\in[d]$ set
\[
\sigma_{-i}=\tau_{-\sigma_{i}(0)}\sigma_{i}^{-1}\tau_{\sigma_{i}(0)},
\]
so in particular 
\[
\sigma_{-i}(0)=-e_{i},\quad\sigma_{-i}(\hat{\MC})=-e_{i}+\hat{\MC}.
\]

We then define the auxiliary model as in equation (\ref{eq:tracer_generator}),
with $\Sigma=\{\sigma_{\pm1},\dots,\sigma_{\pm d}\}$ and $\hat{c}_{\sigma}(\eta)=\One_{\hat{\MC}\text{ is empty}}$
for all $\sigma\in\Sigma$. It is indeed reversible with respect to
$\mu$, and all rates are bounded by $1$ (as required by Lemma \ref{lem:D_comparison}).
\begin{lem}
\label{lem:Ds_auxiliary}Consider the auxiliary model defined above.
Then for all $u\in\R^{d}$
\[
u\cdot\hat{D}_{s}u=\frac{1}{2}q^{\left|\hat{\MC}\right|}\norm u^{2}.
\]
\end{lem}

\begin{proof}
Start the dynamics with a configuration $\eta_{0}$ drawn from $\mu_{0}$
and tracer at the origin.

Assume $\hat{\MC}$ is empty for $\eta_{0}$. Then the entire cluster
$\hat{\MC}\cup\{0\}$ performs a simple random walk, independently
of the initial configuration. This is because initially all rates
are $1$, and in each transition the tracer moves together with $\hat{\MC}$,
meaning that all rates remain $1$. 

On the other hand, if $\hat{\MC}$ is not empty initially, then the
configuration is blocked, and the tracer remain at the origin forever.
Hence, denoting the tracer position at time $t$ by $z_{t}$,
\begin{align*}
u\cdot\hat{D}_{s}u & =\lim_{t\to\infty}\frac{1}{2t}\E\left((u\cdot z_{t})^{2}\right)=\lim_{t\to\infty}\frac{1}{2t}\E\left((u\cdot z_{t})^{2}\One_{\hat{\MC}\text{ is empty for }\eta_{0}}\right)\\
 & =\frac{1}{2}\norm u^{2}\,\mu(\hat{\MC}\text{ is empty for }\eta_{0}).\qedhere
\end{align*}
\end{proof}

\subsubsection{The multistep move}

In this section we construct the multistep moves allowing us to move
the tracer together with an empty cluster $\hat{\MC}$.

Fix a mobile cluster $\MC$ and $\MClen>0$ such that the translation
and exchange moves exist. We define 
\[
\hat{\MC}=\{-e_{1}\}\cup\left((\MClen+2)e_{1}+\MC\right).
\]

\begin{claim}
\label{claim:hopping_move}There exists a $T$-step move $\text{Hop}=((\eta_{t}),(x_{t}),(e_{t}))$,
which we call the vacancy hopping move, such that:
\begin{enumerate}
\item $\Dom\text{Hop}=\left\{ \eta:\text{\ensuremath{\eta(0)=1\text{ and }}}\hat{\MC}\text{ is empty}\right\} $.
\item $\text{Hop}$ is a deterministic move, compatible with the cyclic
permutation $\sigma_{\text{H}}=(e_{1},e_{1}+e_{2},e_{2},-e_{1}+e_{2},-e_{1})$.
\item For all $t$, at least one of the two sites $x_{t}$ or $x_{t}+e_{t}$
must be empty.
\end{enumerate}
\end{claim}

\begin{proof}
We will construct $\text{Hop}$ as a composition of several moves.
First, we use translation moves in order to bring the mobile cluster
to $-e_{1}-\MClen e_{2}+\MC$:
\begin{multline*}
M_{1}=\text{Tr}_{2}(-(\MClen+1)e_{2}-e_{1}+\MC)\circ\text{Tr}_{-1}(-(\MClen+1)e_{2}+\MC)\circ\dots\circ\text{Tr}_{-1}(-(\MClen+1)e_{2}+(\MClen+2)e_{1}+\MC)\\
\circ\text{Tr}_{-2}(-\MClen e_{2}+(\MClen+2)e_{1}+\MC)\circ\dots\circ\text{Tr}_{-2}((l+2)e_{1}+\MC).
\end{multline*}
We emphasize that, for each of these translation $\text{Tr}(x+\MC)$,
the sites $-e_{1},-e_{1}+e_{2}$ are outside $x+[-\MClen,\MClen]$,
hence untouched by the move. Also, the translation move is deterministic,
and since adding vacancies to a configuration in $\Dom\text{Tr}$
keeps it in $\Dom\text{Tr}$, we may assume that all transitions involve
at least one empty site.

Next, we exchange $-e_{1}$ and $-e_{1}+e_{2}$:
\[
M_{2}=\text{Ex}_{2}(-e_{1}-\MClen e_{2}+\MC),
\]
and move the mobile cluster back to $(\MClen+2)e_{1}+\MC$.
\[
M_{3}=M_{1}^{-1}.
\]
So far, we obtain a move $M_{3}\circ M_{2}\circ M_{1}$ with the associated
permutation $(-e_{1},-e_{1}+e_{2})$.

Next, we move the cluster, exchange $-e_{1}+e_{2}$ with $e_{2}$
and the move it back:
\begin{align*}
M_{4} & =\text{Tr}_{-1}((l+1)e_{1}+e_{2}+\MC)\circ\text{Tr}_{-1}((\MClen+2)e_{1}+e_{2}+\MC)\circ\text{Tr}_{2}((\MClen+2)e_{1}+\MC),\\
M_{5} & =\text{Ex}_{-1}(\MClen e_{1}+e_{2}+\MC),\\
M_{6} & =M_{4}^{-1}.
\end{align*}
This results in a move $M_{6}\circ M_{5}\circ M_{4}$ associated to
the permutation $(-e_{1}+e_{2},e_{2})$.

In the same manner we construct a move $M_{7}$ associated with $(e_{2},e_{1}+e_{2})$
and a move $M_{8}$ associated with $(e_{1}+e_{2},e_{1})$.

We end up with the desired multistep move $\text{Hop}=M_{8}\circ M_{7}\circ M_{6}\circ M_{5}\circ M_{4}\circ M_{3}\circ M_{2}\circ M_{1}$.
\end{proof}
\begin{claim}
\label{claim:Ds_move_1}There exists a permutation $\sigma_{1}$ and
a move $M_{\sigma_{1}}$ such that:
\begin{enumerate}
\item $\Dom M_{\sigma_{1}}=\left\{ \eta:\text{\ensuremath{\eta(0)=1\text{ and }}}\hat{\MC}\text{ is empty}\right\} $.
\item $M_{\sigma_{1}}$ is deterministic, compatible with $\sigma_{1}$.
\item $\sigma_{1}(0)=e_{1}$ and $\sigma_{1}(\hat{\MC})=e_{1}+\hat{\MC}$.
\item For all $t$, at least one of the two sites $x_{t}$ or $x_{t}+e_{t}$
must be empty.
\end{enumerate}
\end{claim}

\begin{proof}
The move $M_{\sigma_{1}}$ is given by 
\[
M_{\sigma_{1}}=\text{Tr}_{1}((\MClen+2)e_{1}+\MC)\circ\text{Tr}_{1}((\MClen+1)e_{1}+\MC)\circ\text{Ex}_{-1}((\MClen+1)e_{1}+\MC)\circ\text{Tr}_{-1}((\MClen+2)e_{1}+\MC)\circ\text{Hop}.
\]
\end{proof}
So far, we constructed the permutation $\sigma_{1}$ defining the
auxiliary model, and the move $M_{z_{0},\sigma_{1}}$ required in
Hypothesis \ref{hyp:auxiliary_move} (for $z_{0}=0$ hence for all
$z_{0}$). This gives us automatically $\sigma_{-1}=\tau_{-e_{1}}\sigma_{1}^{-1}\tau_{e_{1}}$,
and the move $M_{e_{1},\sigma_{-1}}=M_{0,\sigma_{1}}^{-1}$, which
provides $M_{z_{0},\sigma_{-1}}$ for all $z_{0}$.

In order to propagate in other directions, we use the following claim:
\begin{claim}
\label{claim:Ds_move}For and $\alpha\in[1]$, there exists a permutation
$\sigma_{\alpha}$ and a move $M_{\sigma_{\alpha}}$ such that:
\begin{enumerate}
\item $\Dom M_{\sigma_{\alpha}}=\left\{ \eta:\text{\ensuremath{\eta(0)=1\text{ and }}}\hat{\MC}\text{ is empty}\right\} $.
\item $M_{\sigma_{\alpha}}$ is deterministic, compatible with $\sigma_{\alpha}$.
\item $\sigma_{\alpha}(0)=e_{\alpha}$ and $\sigma_{\alpha}(\hat{\MC})=e_{\alpha}+\hat{\MC}$.
\item For all $t$, at least one of the two sites $x_{t}$ or $x_{t}+e_{t}$
must be empty.
\end{enumerate}
\end{claim}

\begin{proof}
Claim \ref{claim:Ds_move_1} shows the case $\alpha=1$.

The construction for $\alpha\neq1$ is similar to the previous claims.
Start by exchanging $-e_{1}$ with $-e_{\alpha}$ (in the exact same
manner as the move $M_{6}\circ M_{5}\circ M_{4}\circ M_{3}\circ M_{2}\circ M_{1}$
in the proof of Claim \ref{claim:hopping_move}). Then translate the
mobile cluster from $(\MClen+2)e_{1}+\MC$ to $(\MClen+2)e_{\alpha}+\MC$.
This brings us to the same setting as Claim \ref{claim:Ds_move_1},
where the direction $1$ is replaced by $\alpha$. We may then use
the same construction in order to move $\{0,-e_{\alpha}\}\cup\left((\MClen+2)e_{\alpha}+\MC\right)$
one step in the direction $e_{\alpha}$. Finally, move the mobile
cluster back from $(\MClen+3)e_{\alpha}+\MC$ to $(\MClen+2)e_{1}+e_{\alpha}+\MC$
and the vacancy at $0$ to $e_{\alpha}-e_{1}$.
\end{proof}
Theorem \ref{thm:self_diffusion} then follows from Claim \ref{claim:Ds_move},
Lemma \ref{lem:Ds_comparison}, and Lemma \ref{lem:Ds_auxiliary}.
\qed
\section{Questions}
\begin{itemize}[leftmargin=*]
\item The proofs given here show polynomial divergence of time scales as
$q$ tends to $0$. Is it possible to identify the exact exponent
of this divergence?
\item What is the qualitative behavior of the different quantities described
here when changing $q$? Are they continuous? Smooth? We expect them
to be monotone (since decreasing $q$ should ``slow down'' the system),
but the nonattractivity of the model makes it difficult to prove.
\item Variational formulas can also be used to approximate different quantities,
and not just find bounds---consider, for example, the diffusion coefficient
$D$. We may define, for $\Lambda\subset\Z^{d}$, 
\[
u\cdot D^{(\Lambda)}u=\frac{1}{2q(1-q)}\min_{f}\,\mu\left[\sum_{\alpha=1}^{d}c_{0,e_{\alpha}}\left(u\cdot e_{\alpha}(\eta(0)-\eta(e_{\alpha}))+\sum_{x}\grad_{0,e_{\alpha}}\tau_{x}f\right)^{2}\right],
\]
where the minimum is taken over functions $f:\{0,1\}^{\Lambda}\to\R$.
Then $D=\lim_{\Lambda\to\Z^{d}}D^{(\Lambda)}$.

\cite{AritaKrapivskyMallick2018} evaluated this minimum, obtaining
(nonrigorously) an approximate expression for $D$ of the Kob-Andersen
model, which is a cooperative kinetically constrained lattice gas.
In their case, as $q$ tends to $0$, larger and larger boxes $\Lambda$
must be taken in order to have a good approximation of $D$. We know
that since any finite $\Lambda$ gives $D^{(\Lambda)}$ polynomial
in $q$, and for the Kob-Andersen model the diffusion coefficient
decays superpolynamially.

In noncooperative models, the decays is polynomial, 
so one may hope that a finite box $\Lambda$ could provide
a good approximation of $D$ for all $q$. For the model in Example
\ref{exa:1d} an empty $\Lambda$ already gives the correct diffusion
coefficient up to a factor $2$. What happens in other noncooperative
models? Can we say that $D/D^{(\Lambda)}\to1$ uniformly in $q$?
\item Extend Theorem \ref{thm:gap_closed} to models satisfying Hypothesis \ref{hyp:MC_in_ergodic_component}
in all dimensions, or more generally to all noncooperative models.
\item Given the positivity of the diffusion coefficient (Theorem \ref{thm:D}),
it is natural to conjecture convergence to the hydrodynamic limit
of all noncooperative kinetically constrained lattice gases. Can we
show it for models other than the one studied in \cite{GoncalvesLandimToninelli}?
Proving convergence for nongradient models (e.g. the model in Example
\ref{exa:1d}) is an interesting (and challenging) problem.
\item We expect the equilibrium fluctuations to converge to a Gaussian field
(see, e.g., \cite[II.2]{Spohn2012IPS}), with the diffusion coefficient
studied in Section \ref{sec:diffusion}. Can this be proven?
\item Studying the diffusivity of cooperative kinetically constrained models.
Results analogous to theorems \ref{thm:gap_reservoir}, \ref{thm:D},
and \ref{thm:self_diffusion} have been shown for the Kob-Andersen
model (\cite{MST19KA,S23KA_HL,BlondelToninell2018KAtagged,ES20KAtagged}).
To the author's knowledge, other cooperative models have not been
studied in the mathematical literature. Can one understand ergodicity
properties of cooperative models? Does ergodicity always imply diffusivity?
How do typical time scales diverge near criticality?
\end{itemize}

\bibliographystyle{plain}
\bibliography{noncooperative}

\end{document}